\documentclass[hidelinks, 10pt]{article}
\usepackage{graphicx} 
\usepackage[utf8]{inputenc}
\usepackage[T1]{fontenc} 
\usepackage{float} 
\usepackage[french]{babel} 
\usepackage{comment} 
\usepackage{algorithm2e}
\usepackage{amsmath,amsthm,amssymb}
\usepackage{caption}
\usepackage{hyperref}

\newtheorem{theorem}{Théorème}[section]

\newtheorem{lemma}[theorem]{Lemme}

\newtheorem{proposition}{Proposition}[section]

\theoremstyle{definition}
\newtheorem{definition}{Définition}[section]

\theoremstyle{remark}
\newtheorem*{remark}{Remarque}

\newcommand{\inv}{^{-1}}
\newcommand{\bb}[1]{\mathbb{#1}}

\newcommand{\atan}{\mathrm{atan}}

\newcommand{\s}[1]{\mathcal{S}^{#1}}

\newcommand{\eqCP}{\sim_{CP}}
\newcommand{\eqCM}{\sim_{CM}}
\newcommand{\eqCT}{\sim_{CT}}
\newcommand{\eqCol}{\sim_{col}}
\newcommand{\eqSym}{\sim_{sym}}

\title{\textbf{Polygones fondamentaux et plongements dans l'espace euclidien : ruban de Möbius, tore, et plan projectif}}
\author{Anthony Fraga}
\date{}

\begin{document}

\maketitle
\begin{abstract}
Dans le cadre de la topologie algébrique, il est courant de représenter des surfaces par des polygones fondamentaux. Cette représentation est intrinsèque, mais nécessite d'identifier des points par une relation d'équivalence. Au contraire, un plongement dans un espace euclidien n'est pas intrinsèque, mais ne demande pas d'égaliser des points. Dans cet article, nous étudierons le ruban de Möbius, le tore de dimension 2, et le plan projectif réel. Plus précisément, on construit des homéomorphismes explicites, ainsi que leurs réciproques, entre les polygones fondamentaux et des surfaces de l'espace euclidien : de dimension 3 pour le ruban de Möbius et le tore, et 4 pour le plan projectif. Si les trois plongements sont bien connus, nous n'avons pas connaissance de formules explicites pour leur réciproques.
\end{abstract}

\section{Préliminaires}

Dans le cadre de la rigueur du présent article, et par la forte présence d'espaces quotient, il semble nécessaire de présenter quelques résultats fondamentaux sur ceux-ci.

\subsection{Propriété universelle du quotient topologique}

Étant donné un ensemble $X$ muni d'une relation d'équivalence $\sim$, on note pour $x\in X$ sa \emph{classe d'équivalence} $[x]=\{x'\in X, x\sim x'\}$. L'ensemble des classes d'équivalences, appelé \emph{ensemble quotient}, se note $X/\!\sim\ =\{[x],x\in X\}$. On note $\pi:X\to X/\!\sim$ la \emph{surjection canonique} induite par $X$ et $\sim$.

\begin{definition}
Soit $X$ un espace topologique, muni d'une relation d'équivalence $\sim$. On appelle \emph{topologie quotient} sur $X/\!\sim$, induite par $X$ et $\sim$, la topologie définie par :\[\forall\, U\!\subset\! X/\sim,\qquad U\ ouvert\Leftrightarrow\pi\inv(U)\ ouvert\ dans \ X.\]
\end{definition}

\begin{theorem}[Propriété universelle de l'espace quotient]\label{th:quotient}
Soient $X,Y$ deux espaces topologiques, et~$\sim$ une relation d'équivalence sur $X$. Pour $f:X\to Y$ une application continue vérifiant $x\sim x'\Rightarrow f(x)=f(x')$, il existe une unique application continue $\overline{f}:X/\!\sim\ \longrightarrow Y$, telle que $f=\overline{f}\circ \pi$.

De plus, si $f$ vérifie, pour tout $x,x'$ éléments de $X$, $f(x)=f(x')\Rightarrow x\sim x'$, alors $\overline{f}$ est injective.
\end{theorem}
\begin{proof}
En reprenant les notations de l'énoncé, il nous faut montrer que l'application $\overline{f}$ qui à $[x]$ renvoie $f(x)$, pour $x\in X$, est bien définie. Une telle application vérifie $f=\overline{f}\circ \pi$, car nous avons : \[\forall x\in X,\quad \overline{f}\big(\pi(x)\big)=\overline{f}([x])=f(x).\]

Soient $x,x'\in X$ tel que $x\sim x'$. Cela implique que $f(x)=f(x')$. Il s'ensuit alors l'égalité suivante : \[
\overline{f}([x])=f(x)=f(x')=\overline{f}([x']).\]Ainsi, $\overline{f}$ ne dépend pas du représentant, mais uniquement de la classe. Cela nous permet donc de conclure que $\overline{f}$ est bien définie. 

\bigskip Montrons que cette application est continue. Soit $V$ un ouvert de $Y$, nous voulons montrer que $\overline{f}\inv(V)$ est un ouvert de $X/\sim$. Or, par définition de l'espace quotient, cela revient à vérifier que $\pi\inv\left(\overline{f}\inv(V)\right)$ est un ouvert de $X$. En remarquant l'égalité suivante, \[\pi\inv\left(\overline{f}\inv(V)\right)=\left(\overline{f}\circ\pi\right)\inv(V)=f\inv(V),\]nous pouvons, par continuité de $f$, conclure que $\overline{f}\inv(V)$ est bien un ouvert de~$X/\sim$, ce qui fait de $\overline{f}$ une application continue.

\bigskip Montrons désormais l'unicité d'une telle application. Considérons $g$ une application vérifiant également $f=g\circ\pi$. Montrons que $g=\overline{f}$.

Soit $\alpha\in X/\!\sim$, et $x\in\pi^{-1}(\alpha)$. Nous avons l'égalité : \[g(\alpha)=g\big(\pi(x)\big)=f(x)=\overline{f}\big(\pi(x)\big)=\overline{f}(\alpha),\]ce qui nous permet de déduire que $g=\overline{f}$. On en conclut alors que $\overline{f}$ est unique.

\bigskip Enfin, supposons que l'application $f$ vérifie, pour tout $x,x'$ éléments de $X$, l'implication $f(x)=f(x')\Rightarrow x\sim x'$. Montrons alors que l'application $\overline{f}$ est injective.

Soient $[x],[x']\in X/\sim$. On a les implications suivantes : \[\overline{f}([x])=\overline{f}([x'])\Longrightarrow f(x)=f(x')\Longrightarrow x\sim x'\Longrightarrow [x]=[x'],\]ce qui nous permet de conclure que $\overline{f}$ est injective.
\end{proof}

\begin{definition}
En reprenant les notations, une telle application $f$ est dite \emph{compatible} avec la relation d'équivalence $\sim$.
\end{definition}

\subsection{Stratégie pour construire les homéomorphismes}

Comme annoncé dans le résumé, on souhaite relier entre elles les caractérisations du ruban de Möbius, du tore et du plan projectif réel en termes de polygones fondamentaux à leurs plongements dans l'espace euclidien. Formellement, cela revient à construire des homéomorphismes.

\begin{definition}
Soient $X,Y$ deux espaces. Un \emph{homéomorphisme} entre $X$ et $Y$ est une application $f$ bijective et continue, telle que son inverse $f\inv$ soit également continue. On dit dans ce cas que les espaces $X$ et $Y$ sont \emph{homéomorphes}.
\end{definition}

La propriété universelle énoncée dans la section précédente nous permet de construire une application quotient injective et continue (lorsque son application de base est définie comme il se doit). Par la suite, nous montrerons que cette application admet également une inverse à droite. Avec la proposition qui suit, nous en déduisons que notre application est bijective, et d'inverse son inverse à droite. Il ne nous restera plus qu'à montrer la continuité de l'application inverse, pour en conclure que notre application est bien un homéomorphisme.

\begin{proposition}\label{prop:inj+inv=bij}
Une application injective et inversible à droite est bijective, et d'inverse son inverse à droite.
\end{proposition}

\begin{proof}
Soit $f:X\to Y$ une application injective, et soit $g:Y\to X$ une inverse à droite de celle-ci. Pour montrer que l'application $f$ est bijective, il nous suffit alors de montrer qu'elle est surjective.

\bigskip Pour $y\in Y$, on a $f(g(y))=f\circ g(y)=y$, ce qui veut dire que $g(y)$ est un antécédent de $y$. On en conclut dès lors que $f$ est une application surjective.

\bigskip Montrons désormais l'unicité de l'inverse, montrant alors que $g$ est bien aussi inverse à gauche. Soit $h$ une autre inverse de $f$. Comme $h$ est en particulier une inverse à gauche et $g$ une inverse à droite, on a $h=h\circ f\circ g=g$. On en conclut alors que $h=g$.
\end{proof}

\subsection{L'application atan2}

Certaines de nos surfaces dans l'espace euclidien sont paramétrés avec des angles polaires. Pour pouvoir construire une réciproque à cette paramétrisation, nous devons utiliser l'application $\atan2$.

L'application \emph{atan2} est une variante de la fonction $\arctan$, qui prend cette fois deux arguments. Elle permet de redonner l'angle polaire à partir des coordonnées cartésiennes, alors que la fonction $\arctan$ identifie les angles diamétralement opposés.

Cette fonction est définie de la manière suivante : \[\atan2(y,x)=2\arctan\left(\frac{y}{\sqrt{x^2+y^2}+x}\right).\]De manière équivalente, nous avons également la définition suivante :\begin{equation}\label{eq:atan2}
\begin{split}
\forall x\neq 0&\qquad \atan2(0,x)=\left\{\begin{matrix}
0&\ si\ x>0\\ 
\pi &\ si\  x<0
\end{matrix}\right.\\
\forall y\neq 0&\qquad \atan2(y,x)=\left\{\begin{matrix}
\arctan(\frac{y}{x})&\text{si}\ x>0\\ 
\varepsilon(y)\frac{\pi}{2}&\text{si}\ x=0\\
\arctan\left(\frac{y}{x}\right) +\varepsilon(y) \pi&\text{si}\  x<0,
\end{matrix}\right.
\end{split}
\end{equation}avec $\varepsilon:y\mapsto\pm1$ qui renvoie le signe de $y$.

\begin{figure}[H]
    \centering
    \includegraphics[width=0.55\linewidth]{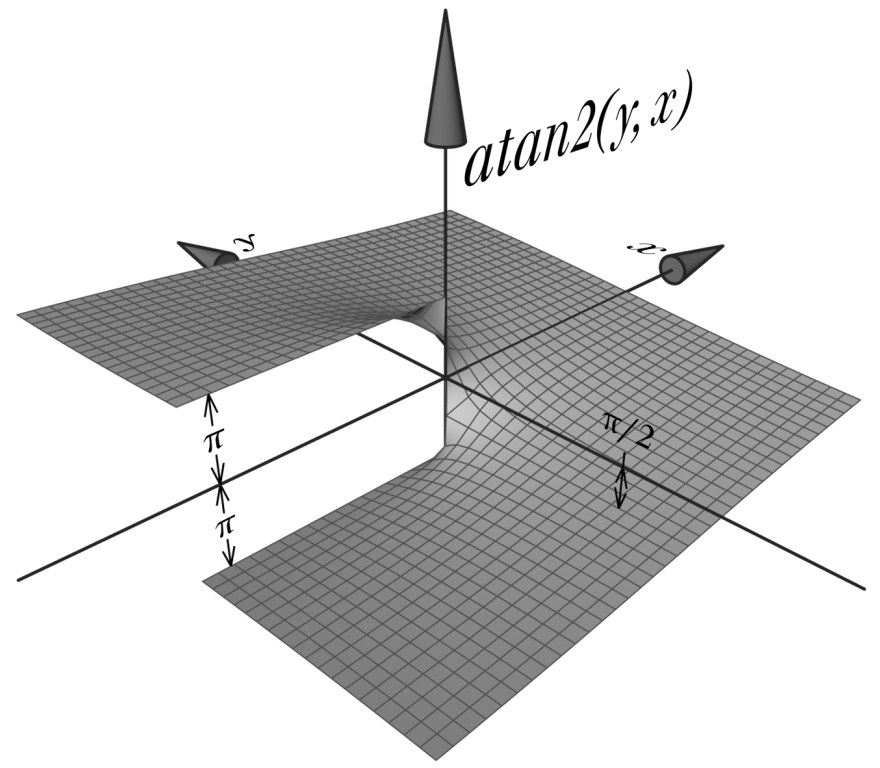}
    \caption{Graphe de l'application application $atan2$}
    \label{fig:graph-atan2}
\end{figure}
Il semble important de remarquer, d'une part avec le graphique, et d'autre part avec la définition, que l'application $\atan2$ n'est pas continue sur le segment~$[-\pi,0[\times\{0\}$. En effet, la limite lorsque $x<0$ et $y<0$ tendent vers 0 est~$-\pi$, alors que la valeur donnée $y=0,x<0$ est $\pi$.

\bigskip Nous pouvons désormais commencer à nous intéresser à nos espaces topologiques.

\newpage

\section{Le ruban de Möbius}
\subsection{Polygone fondamental}

Le premier espace qui nous intéresse est celui qui est appelé \emph{ruban de Möbius}, portant le nom de son inventeur. C'est un complexe cellulaire qui est obtenu par la figure suivante, auquel on identifie un côté par son opposé, dans le sens inverse.

\begin{figure}[H]
    \centering
    \includegraphics[width=0.25\linewidth]{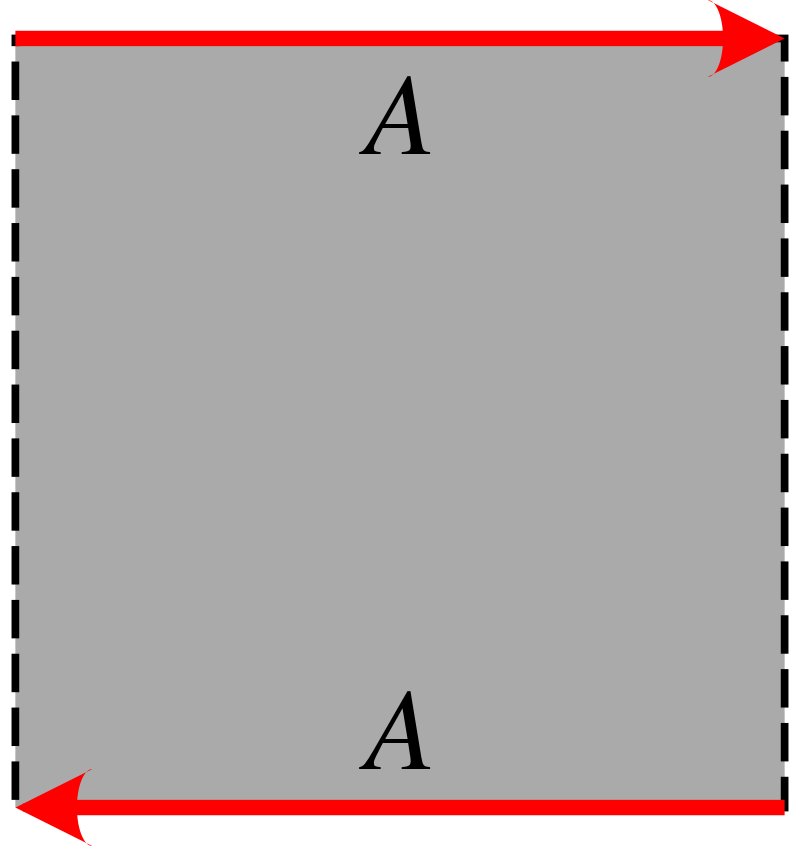}
    \caption{Polygone fondamental du ruban de Möbius}
    \label{fig:mobius-square}
\end{figure}

\begin{definition}
On considère le carré $[-1,1]^2$, ainsi que la relation d'équivalence $\eqCM$ définie par : \[\forall(t,v),(t',v')\in[-1,1]^2,\quad (t,v)\eqCM(t',v')\Leftrightarrow \left\{\begin{matrix}
(t',v')=\pm(t,v)& \text{si }v=\pm1\\
(t',v')=(t,v) & \text{sinon.}
\end{matrix}\right.\]On définit ainsi l'espace quotient $C_M:=[-1,1]^2/\eqCM$.
\end{definition}

\subsection{Plongement dans l'espace euclidien}

Ce complexe cellulaire peut être plongé dans l'espace euclidien $\mathbb{R}^3$, par la définition qui suit. Intuitivement, nous pouvons imaginer partir d'une feuille de papier, pour recoller deux côtés opposés, dans le sens suivant les flèches du polygone fondamental. On obtient alors une bande recollée avec une torsion d'un demi-tour, ce qui donne la figure ci-dessous.

\begin{figure}[H]
    \centering
    \includegraphics[width=0.35\linewidth]{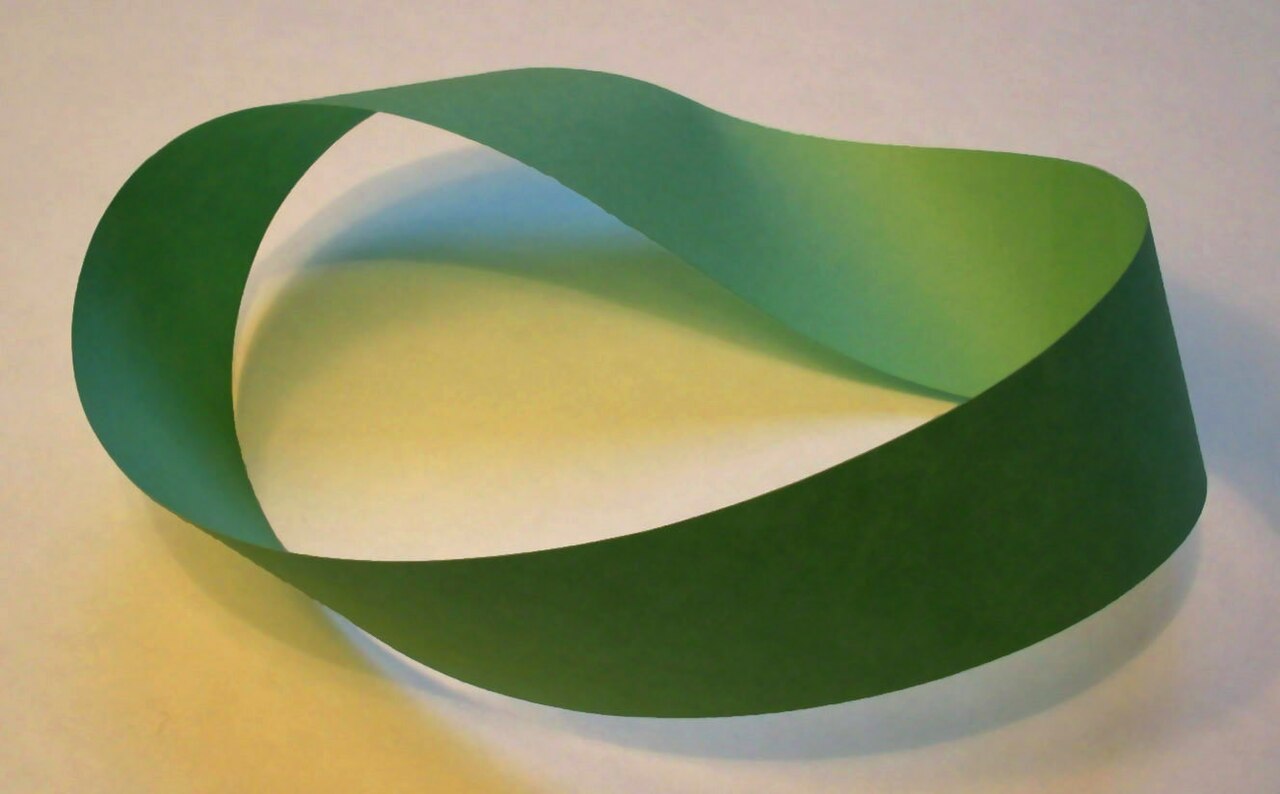}
    \caption{Ruban de Möbius obtenu à partir d'une feuille}
    \label{fig:mobius-strip}
\end{figure}

\begin{definition}\label{def:mobius-R3}
On considère la surface paramétrée suivante dans $\bb{R}^3$ : \begin{equation}\label{eq:mobius}
\left\{\begin{matrix}
x(t,v)&=&\left(1+\frac{t}{2}\cos\frac{v\pi}{2} \right )\cos (v\pi)\\ 
y(t,v)&=&\left(1+\frac{t}{2}\cos\frac{v\pi}{2} \right )\sin (v\pi)\\ 
z(t,v)&=&\frac{t}{2}\sin\frac{v\pi}{2},
\end{matrix}\right.
\end{equation}
avec $v\in[-1,1],t\in[-1,1]$.

Nous noterons l'application $m:(t,v)\mapsto\big((x(t,v),y(t,v),z(t,v)\big)$ permettant la paramétrisation, ainsi que la surface sur $\bb{R}^3$ contenant tout les points de la paramétrisation $M=im(m)=\{m(t,v),t\in[-1,1],v\in[-1,1]\}$.

Cette surface admet comme équation cartésienne la suivante : \begin{equation}\label{eq:mobius-cartesien}
y(x^2+y^2+z^2-1)-2z(x^2+y^2+x)=0.
\end{equation}
\end{definition}

\begin{remark}
Le paramètre $t$ permet de se déplacer sur la largeur du ruban. Le paramètre $v$ permet le déplacement sur la longueur du ruban.
\end{remark}

\begin{proposition}\label{prop:mobius-z-value}
Soit $(x,y,z)\in M$, avec $y\neq0$. On peut écrire $z$ en fonction des deux autres variables, de la manière suivante : $$z=\frac{(x^2+y^2+x)\pm(x+1)\sqrt{x^2+y^2}}{y}.$$
\end{proposition}
\begin{proof}
Réécrivons l'équation cartésienne \eqref{eq:mobius-cartesien} de la surface $M$, de façon a faire apparaître une équation du second degré en $z$ : \[yz^2-2z(x^2+y^2+x)+(yx^2+y^3-y)=0.\]On peut alors calculer le discriminant : \[
\begin{split}
\Delta&=4(x^2+y^2+x)^2-4y(yx^2+y^3-y)\\
\frac{\Delta}{4}&=x^4+y^4+x^2+2x^2y^2+2x^3+2xy^2-y^2x^2-y^4+y^2\\
&=(x^4+2x^3+x^2)+y^2(x^2+2x+1)\\
&=(x^2+y^2)(x+1)^2.
\end{split}\] Avec la paramétrisation \eqref{eq:mobius} de $x$, on peut voir qu'il est minoré par $-1$, obtenu lorsque $v=\pm1$. Les deux facteurs de $\Delta$ sont donc positifs, il existe ainsi des solutions réelles de l'équation du second degré, donné par : \[z=\frac{(x^2+y^2+x)\pm(x+1)\sqrt{x^2+y^2}}{y}.\]
\end{proof}

\begin{proposition}\label{prop:mobius-equiv}
L'application $m$ vérifie l'équivalence suivante : \begin{equation*}
\forall(t,v),(t',v')\in[-1,1]^2,\quad (t,v)\eqCM(t',v')\Leftrightarrow m(t,v)=m(t',v')
\end{equation*}
\end{proposition}
\begin{proof}
Montrons dans un premier temps l'implication $(\Rightarrow)$. Soient $(t,v),(t',v')$ deux couples de $[-1,1]^2$, tels que $(t,v)\eqCM(t',v')$. Montrons que l'on a bien l'égalité $m(t,v)=m(t',v')$.\begin{itemize}
    \item Dans le cas où $v\neq\pm1$, c'est évident, puisque $(t',v')=(t,v)$.

    \item En supposant $v=\pm1$, on a $(t',v')=\pm(t,v)$. Il suffit alors de traiter le cas où le signe est négatif, en remarquant que pour $t\in[-1,1]$, \[\begin{split}
m(t,-1)&=\left\{\begin{matrix}
x(t,-1)&=&\left(1+\frac{t}{2}\cos\frac{-\pi}{2} \right )\cos(-\pi)\\ 
y(t,-1)&=&\left(1+\frac{t}{2}\cos\frac{-\pi}{2} \right )\sin(-\pi)\\ 
z(t,-1)&=&\frac{t}{2}\sin\frac{-\pi}{2}
\end{matrix}\right.\\
&=\left\{\begin{matrix}
x(t,-1)&=&-1\\ 
y(t,-1)&=&0\\ 
z(t,-1)&=&-\frac{t}{2},
\end{matrix}\right.
\end{split}\]ainsi que, \[\begin{split}
m(-t,1)&=\left\{\begin{matrix}
x(-t,1)&=&\left(1-\frac{t}{2}\cos\frac{\pi}{2} \right )\cos(\pi)\\ 
y(-t,1)&=&\left(1-\frac{t}{2}\cos\frac{\pi}{2} \right )\sin(\pi)\\ 
z(-t,1)&=&-\frac{t}{2}\sin\frac{\pi}{2}
\end{matrix}\right.\\
&=\left\{\begin{matrix}
x(-t,1)=&-1\\ 
y(-t,1)=&0\\ 
z(-t,1)=&-\frac{t}{2},
\end{matrix}\right.
\end{split}\]on peut donc en déduire que $m(t,-1)=m(-t,1)$. On en conclut, pour $v=\pm1$, que l'on a bien $m(t',v')=m(t,v)$, avec $(t', v')=\pm(t,v)$.
\end{itemize}
Nous pouvons alors conclure que l'application $m$ est compatible avec la relation d'équivalence $\eqCM$.

\bigskip Désormais, montrons l'implication réciproque $(\Leftarrow)$.

Soient $(t,v),(t',v')\in[-1,1]^2$. Étudions l'équation $m(t,v)=m(t',v')$. \begin{itemize}
\item Supposons $t=0$. Ainsi, l'équation paramétrique nous donne : \[m(0,v)=m(t',v')\Longrightarrow\left\{\begin{matrix}
\cos(v\pi)&=&\left(1+\frac{t'}{2}\cos(\frac{v'\pi}{2})\right)\cos(v'\pi)\\ 
\sin(v\pi)&=&\left(1+\frac{t'}{2}\cos(\frac{v'\pi}{2})\right)\sin(v'\pi)\\ 
0&=&\frac{t'}{2}\sin(\frac{v'\pi}{2}).
\end{matrix}\right.\]En sommant les carrés des deux premières équations, on obtient l'égalité~$1=\big(1+\frac{t'}{2}\cos(\frac{v'\pi}{2})\big)^2$, c'est à dire $0=\frac{t'}{2}\cos(\frac{v'\pi}{2})$. Avec la dernière équation, on en déduit donc que $t'=0$. En reprenant les deux premières équations, on a $\cos(v\pi)=\cos(v'\pi)$ et $\sin(v\pi)=\sin(v'\pi)$ respectivement. Du fait que ses applications soient $2\pi$-périodiques, on en déduit soit~$v=v'$, ou soit $v'=\pm v$ si $v=\pm1$. Dans les deux cas, nous avons bien~$(t,v)\eqCM(t',v')$.
\item Supposons $t\neq 0$ et $v=\pm\frac{1}{2}$. Dans ce cas, l'équation paramétrique nous donne : \[m\left(t,\pm\frac{1}{2}\right)=m(t',v')\Longrightarrow\left\{\begin{matrix}
0&=&\left(1+\frac{t'}{2}\cos(\frac{v'\pi}{2})\right)\cos(v'\pi)\\ 
\pm\left(1+\frac{t}{2\sqrt{2}}\right)&=&\left(1+\frac{t'}{2}\cos(\frac{v'\pi}{2})\right)\sin(v'\pi)\\ 
\pm\frac{t}{\sqrt{2}}&=&t'\sin\left(\frac{v'\pi}{2}\right).
\end{matrix}\right.\]La première équation nous dit que l'un des deux facteurs est nul, et la seconde nous permet d'en déduire que c'est le terme cosinus. Ainsi, on a~$\cos(v'\pi)=0$, ce qui veut dire que $v'\in\{-\frac{1}{2},\frac{1}{2}\}=\{\pm v\}$. Observons les cas de plus près : \begin{itemize}
    \item Si $v'=-v$, la dernière équation nous donne $t'=-t$, par antisymétrie du sinus. Or, dans ce cas, la seconde équation nous donnerait : \[1+\frac{t}{2\sqrt{2}}=-\left(1+\frac{t'}{2}\frac{1}{\sqrt{2}}\right),\]ou encore :\[1+\frac{t}{2\sqrt{2}}=-1+\frac{t}{2\sqrt{2}}.\]Cela engendrerait $1=-1$, ce qui est impossible.
    \item Si $v'=v$, la troisième équation donne $t'=t$, et donc $(t',v')=(t,v)$.
\end{itemize}
Donc, nous pouvons en conclure que $(t,v)\eqCM(t',v')$.
\item Supposons que $t\neq0$ et $v\neq\pm\frac{1}{2}$. Ainsi, l'équation paramétrique donne : \[m(t,v)=m(t',v')\Rightarrow\left\{\begin{matrix}
\left(1+\frac{t}{2}\cos\frac{v\pi}{2} \right )\cos (v\pi)&=&\left(1+\frac{t'}{2}\cos\frac{v'\pi}{2} \right )\cos (v'\pi)\\ 
\left(1+\frac{t}{2}\cos\frac{v\pi}{2} \right )\sin (v\pi)&=&\left(1+\frac{t'}{2}\cos\frac{v'\pi}{2} \right )\sin (v'\pi)\\ 
\frac{t}{2}\sin\frac{v\pi}{2}&=&\frac{t'}{2}\sin\frac{v'\pi}{2}.
\end{matrix}\right.\]En divisant la seconde équation par la première, nous obtenons l'égalité suivante : $$\tan(v\pi)=\tan(v'\pi).$$ Sachant que la fonction tangente est $\pi$-périodique, nous en déduisons alors que $v'\pi=v\pi+\varepsilon\pi$, c'est à dire $v'=v+\varepsilon$, avec $\varepsilon\in\{0,\pm1,\pm2\}$. Distinguons tout les cas : \begin{itemize}
    \item Si $\varepsilon=0$, nous avons $v'=v$, ce qui nous force à avoir également $t'=t$ pour obtenir une solution de l'équation. Ainsi, on a $(t',v')=(t,v)$, ce qui nous permet de dire que $(t,v)\eqCM(t',v')$.
    \item Si $\varepsilon=\pm1$, cela veut dire que l'on a $\cos(v'\pi)=-\cos(v\pi)$, ainsi que~$\cos\big(\frac{v'\pi}{2}\big)=\pm\sin\big(\frac{v\pi}{2}\big)$. Nous avons alors : \[\begin{split}
    &\left(1+\frac{t}{2}\cos\frac{v\pi}{2} \right )\cos (v\pi)=\left(1\pm\frac{t'}{2}\sin\frac{v\pi}{2} \right )(-\cos (v\pi))\\
    &\Longrightarrow 1+\frac{t}{2}\cos\left(\frac{v\pi}{2}\right)=-1\pm\frac{t'}{2}\sin\left(\frac{v\pi}{2}\right)\quad\text{car }v\neq\pm\frac{1}{2}
    \end{split}\]Or, le terme de gauche est minoré par $\frac{1}{2}$, et celui de droite est majoré par $-\frac{1}{2}$. Il et donc impossible de trouver une solution à cette équation.
    \item Dans le cas où $\varepsilon=\pm2$, on a $v=\pm1$ et $v'=-v$. Également, on a~$\sin\big(\frac{v'\pi}{2}\big)=-\sin\left(\frac{v\pi}{2}\right)$, ce qui nous donne $t'=-t$ avec la dernière équation. On en déduit donc que $(t',v')=-(t,v)$, et par extension que~$(t,v)\eqCM(t',v')$.
\end{itemize}
\end{itemize}Nous pouvons finalement en conclure que $m$ vérifie $m(t,v)=m(t',v')$ si et seulement si $(t,v)\eqCM~(t',v')$.
\end{proof}

Par la propriété universelle du quotient \ref{th:quotient}, nous savons qu'il existe une unique application injective et continue $\overline{m}:C_M\to M$ tel que $\overline{m}=m\circ\pi_m$, avec~$\pi_m$ la surjection canonique pour le quotient par $\eqCM$. 

Dans la suite, nous chercherons à construire une inverse de cette application.

\begin{proposition}\label{prop:reciproque-mobius}
L'application suivante est une inverse à droite de l'application $m$ : \[\begin{split}
g_m:M&\longrightarrow [-1,1]^2\\
(x,y,z)&\mapsto\left ( \left\{\begin{matrix}
2(x-1) &\text{si }\atan2(y,x)=0 \\ 
\frac{2z}{\sin\left(\frac{\atan2(y,x)}{2}\right)} & \text{sinon}
\end{matrix}\right.,\quad\frac{1}{\pi} \atan2(y,x) \right ).
\end{split}\]
\end{proposition}
\begin{proof}
Nous voulons montrer l'égalité $m\circ g_m=id_M$. Soit $(x,y,z)\in~M$, pour lequel on choisit $(t,v)\in[-1,1]^2$ tel que $m(t,v)=~(x,y,z)$. Posons enfin~$(t',v')=g_m(x,y,z)$. On veut alors montrer que $(t,v)\eqCM(t',v')$, ce qui suffit car $m\circ g_m(x,y,z)=m(t',v')=m(t,v)=(x,y,z)$.\begin{itemize}
    \item Si $v=0$, alors $m(t,0)=(1+\frac{t}{2},0,0)$. Or, $\atan2(y,x)=\atan2(0,1+\frac{t}{2})=0$, ce qui donne $g_m(1+\frac{t}{2},0,0)=\Big(2\big(1+\frac{t}{2}-1\big),0\Big)=(t,0)$. On peut en déduire que~$(t',v')=(t,0)$, et donc $(t,v)\eqCM(t',v')$.
    \item Si $v=1$, alors $m(t,1)=\big(-1,0,\frac{t}{2}\big)$. Or,  $\atan2(y,x)=\atan2(0,-1)=~\pi$, ce qui donne~$g_m(x,y,z)=g_m\big(-1,0,\frac{t}{2}\big)=\left(\frac{2\frac{t}{2}}{\sin(\frac{\pi}{2})},1\right)=(t,1)$. On en déduit~$(t',v')=~(t,1)$.
    \item Si $v=-1$, alors on a $m(t,-1)=m(-t,1)=\big(-1,0,-\frac{t}{2}\big)$. On obtient donc~$g_m(x,y,z)=g_m\big(-1,0,-\frac{t}{2}\big)=(-t,1)$, d'après le point précédent. Or, du fait que $(t,-1)\eqCM(-t,1)$, on a bien $(t,v)\eqCM(t',v')$.
    \item Enfin si $v\notin\{\pm1,0\}$, on a $\cos\big(\frac{v\pi}{2}\big)\neq0,\sin\big(\frac{v\pi}{2}\big)\neq0$, ainsi que~$\sin(v\pi)\neq0$. Autrement dit, $y\neq0$. Montrons que l'on a $v\pi=v'\pi=\atan2(y,x)$.\begin{itemize}
        \item Si $x=0$, cela veut dire que $\cos(v\pi)=~0$, donc $v\pi=\pm\frac{\pi}{2}$. En remarquant que $\varepsilon(y)=\varepsilon(\sin(v\pi))=\varepsilon(v)$, on a~$v\pi=\varepsilon(y)\frac{\pi}{2}=\atan2(y,0)$. 
        \item Si $x\neq0$, nous pouvons calculer ce qui suit : \[\frac{y}{x}=\frac{\Big(1+\frac{t}{2}\cos\left(\frac{v\pi}{2}\right)\Big)\sin(v\pi)}{\Big(1+\frac{t}{2}\cos\left(\frac{v\pi}{2}\right)\Big)\cos(v\pi)}=\frac{\sin(v\pi)}{\cos(v\pi)}=tan(v\pi).\]Dans le cas où $x>0$, nous avons $v\pi=\arctan\left(\frac{y}{x}\right)=\atan2(y,x)$. Dans le cas où $x<0$, l'angle $v\pi$ n'appartient pas au codomaine de $\arctan$, qui est $]-\frac{\pi}{2},\frac{\pi}{2}[$. Alors $\arctan(\frac{y}{x})$ nous donne son angle opposé~$v\pi\pm\pi$, qui lui est bien dans le codomaine. On en déduit alors que $v\pi=\arctan(\frac{y}{x})\pm\pi=\atan2(y,x)$, le signe devant $\pi$ donné par celui de $y$. En effet, si $y>0$, alors $\arctan(\frac{y}{x})+\pi\in]-\frac{\pi}{2},\frac{\pi}{2}[$ ; et si~$y<0$, alors $\arctan(\frac{y}{x})-\pi\in]-\frac{\pi}{2},\frac{\pi}{2}[$.
    \end{itemize}
    On peut ainsi conclure que l'on a bien $v\pi=v'\pi$. On peut alors noter~$v\pi=~\atan2(y,x)$.
    Il nous reste à montrer que $t'=t$. Il suffit de voir que $z=\frac{t}{2}\sin\left(\frac{v\pi}{2}\right)$, et donc que l'on a bien~$t'=\frac{2z}{\sin\left(\frac{v\pi}{2}\right)}=t$.
\end{itemize}
Nous pouvons en conclure que $(t,v)\eqCM(t',v')$, et donc que $g_m$ est bien une inverse à droite de $m$.
\end{proof}

\begin{proposition}\label{prop:sin-arc}
Pour $x\neq0$, on a $\sin(\arctan(x))=\frac{x}{\sqrt{1+x^2}}$.
\end{proposition}
\begin{proof}
Soit $ABC$ un triangle rectangle en B, de longueurs $AB=1$ et~$BC=x>0$. L'hypoténuse est donc de longueur $AC=\sqrt{1+x^2}$. Nous pouvons calculer le sinus de l'angle en A : \[\sin(\widehat{CAB})=\frac{BC}{AC}=\frac{x}{\sqrt{1+x^2}}.\]De plus la tangente de cet angle vaut $\frac{BC}{AB}=\frac{x}{1}=x$. On peut alors en déduire que $\widehat{CAB}=\arctan(x)$, ce qui nous permet donc de conclure : \[\sin(\arctan(x))=\sin(\widehat{CAB})=\frac{x}{\sqrt{1+x^2}}.\]
\end{proof}

\begin{lemma}\label{lemma:mobius-inv-continuous}
L'application $\pi_m\circ g_m$ est l'inverse de $\overline{m}$, et elle est continue.
\end{lemma}
\begin{proof}
On peut dans un premier temps voir que l'on a : \[\overline{m}\circ\pi_m\circ g_m=m\circ g_m=id_M.\]Cela montre que l'application $\overline{m}$ admet $\pi_m\circ g_m$ comme inverse à droite. 

\bigskip Montrons désormais que l'application $\pi_m\circ g_m$ est continue. Dans un premier temps, comme expliqué lors de l'introduction de l'application $\atan2(y,x)$, l'application n'est pas continue sur le segment $[-\pi,0[\times\{0\}$. Nous avons vu que les deux valeurs pour $y<0$ et $y>0$ sont respectivement $-\pi$ et $\pi$. Or, ces valeurs correspondent à une valeur de $v$ respectivement égale à $-1$ et 1. Sachant que par l'application $\pi_m$, ces valeurs sont identifiées l'une par rapport à l'autre, on peut alors en conclure que cette composante, composée avec $\pi_m$, est bien continue.

\bigskip Il nous reste alors à montrer la continuité de la première composante. Lorsque l'on a~$\atan2(y,x)\neq 0$, cette composante est donnée par $\frac{2z}{\sin\left(\frac{\atan2(y,x)}{2}\right)}$. De part ce que l'on vient de voir, cette application est bien continue lorsque l'on compose avec $\pi_m$. Il nous faut alors montrer que cette application tend bien vers $2(x-1)$ lorsque $\atan2(y,x)$ tend vers 0.

Dire que $\atan2(y,x)$ tend vers 0 revient à dire que $x>0$ et $y$ tend vers 0. Nous supposons alors $x>0$, et nous étudierons la limite pour $y$ qui tend vers 0 de cette application. Par la proposition \ref{prop:mobius-z-value}, on sait que l'on a : $$z=\frac{(x^2+y^2+x)\pm(x+1)\sqrt{x^2+y^2}}{y}.$$Or, si $y=0$ et $x>0$, cela implique que $\sin(v\pi)=0$, donc $v\in\{0,\pm1\}$. De plus, pour $v=\pm1$, on a $x<0$, ce qui contredit notre hypothèse de départ. Ainsi, $v=0$, et donc $z=0$. Dès lors, la valeur de $z$ ici est nécessairement~$z=~\frac{(x^2+y^2+1)-(x+1)\sqrt{x^2+y^2}}{y}$, car pour l'autre valeur, il n'existe pas de limite lorsque $y$ tend vers 0.

Ensuite, nous noterons $Y=\frac{y}{\sqrt{x^2+y^2}+x}$, de telle sorte que l'on souhaite étudier l'application $\frac{2z}{\sin(\arctan(Y))}$, d'après la définition de $\atan2$. Or, avec la proposition précédente, nous avons l'égalité : \begin{equation}
    \forall X\neq0, \qquad \sin(\arctan(X))=\frac{X}{\sqrt{1+X^2}},
\end{equation}ce qui nous permet d'obtenir les égalités suivantes, lorsque $y$ est suffisamment proche de 0. Nous noterons $r=\sqrt{x^2+y^2}$, de telle sorte que l'on ait $Y=\frac{y}{r+x}$, et $z=\frac{(r^2+x)-(x+1)r}{y}$ : \[\begin{split}
\frac{2z}{\sin(\arctan(Y))}&=\frac{2z\sqrt{1+Y^2}}{Y}\\
&=2\frac{(r^2+x)-(x+1)r}{y}\times\frac{1}{Y}(\sqrt{1+Y^2})\\
&=2\frac{(r^2+x)-(x+1)r}{y}\times\left(\frac{1}{Y}\left(1+\frac{Y^2}{2}\right)+o(Y^2)\right)\\
&=2\frac{(r^2+x)-(x+1)r}{y}\times\frac{r+x}{y}\left(1+\frac{y^2}{2(r+x)^2}+o(y^2)\right).
\end{split}\]En utilisant le développement limité en 0 de $\sqrt{1+X}$ sur $r$, on obtient l'approximation $r=x\sqrt{1+\left(\frac{y}{x}\right)^2}=x\left(1+\frac{y^2}{2x^2}\right)+o(y^3)=x+\frac{y^2}{2x}+o(y^3)$. Lorsque l'on utilise ce résultat dans les calculs précédents, on obtient : \[\begin{split}
\frac{2z}{\sin(\arctan(Y))}&=2\frac{(x^2+y^2+x)-(x+1)(x+\frac{y^2}{2x})}{y}\frac{2x+\frac{y^2}{2x}}{y}+o(y^2)\\
&=2\frac{x^2+y^2+x-x^2-x-\frac{y^2}{2}-\frac{y^2}{2x}}{y}\frac{2x+\frac{y^2}{2x}}{y}+o(y^2)\\
&=2\frac{y^2(\frac{1}{2}-\frac{1}{2x})}{y^2}\left(2x+\frac{y^2}{2x}\right)+o(y^2)\\
&=\left(1-\frac{1}{x}\right)2x+o(y^2)\\
&=2x-2+o(y^2)\\
&=2(x-1)+o(y^2).
\end{split}\]Nous pouvons alors en conclure que la limite lorsque $\atan2(x,y)$ tend vers 0, correspond à la valeur donnée lorsque $\atan2(y,x)=0$. Cela veut donc dire que la première composante de $g_m$ est bien continue.

Nous pouvons alors en conclure que l'application $g_m$ est bien continue sur la surface $M$.
\end{proof}

\begin{theorem}\label{th:mobius-homeo}
Le quotient topologique $C_M$ et la surface $M$ de $\bb{R}^3$ sont homéomorphes.
\end{theorem}
\begin{proof}
Le résultat de la proposition \ref{prop:mobius-equiv}, ainsi que la propriété universelle du quotient, impliquent que l'application $\overline{m}$ est injective et continue. Dans le lemme \ref{lemma:mobius-inv-continuous}, nous avons montré que $\pi_m\circ g_m$ était l'inverse à droite de $\overline{m}$, et qu'elle était elle aussi continue. Finalement, avec la proposition \ref{prop:inj+inv=bij}, nous pouvons en déduire que $\overline{m}$ est une application bijective et continue. Son inverse, étant $\pi_m\circ g_m$, est également continue, ce qui nous permet de dire que $\overline{m}$ est un homéomorphisme entre l'espace quotient $C_M$ et la surface $M$ paramétrée de~$\bb{R}^3$.
\end{proof}

\section{Le tore de dimension 2}
\subsection{Polygone fondamental}

Le second espace est le \emph{tore de dimension 2}, que l'on appellera simplement tore par la suite. C'est également un complexe cellulaire obtenu par l'identification des côtés opposés d'un carré, comme sur la figure suivante.

\begin{figure}[H]
    \centering
    \includegraphics[width=0.3\linewidth]{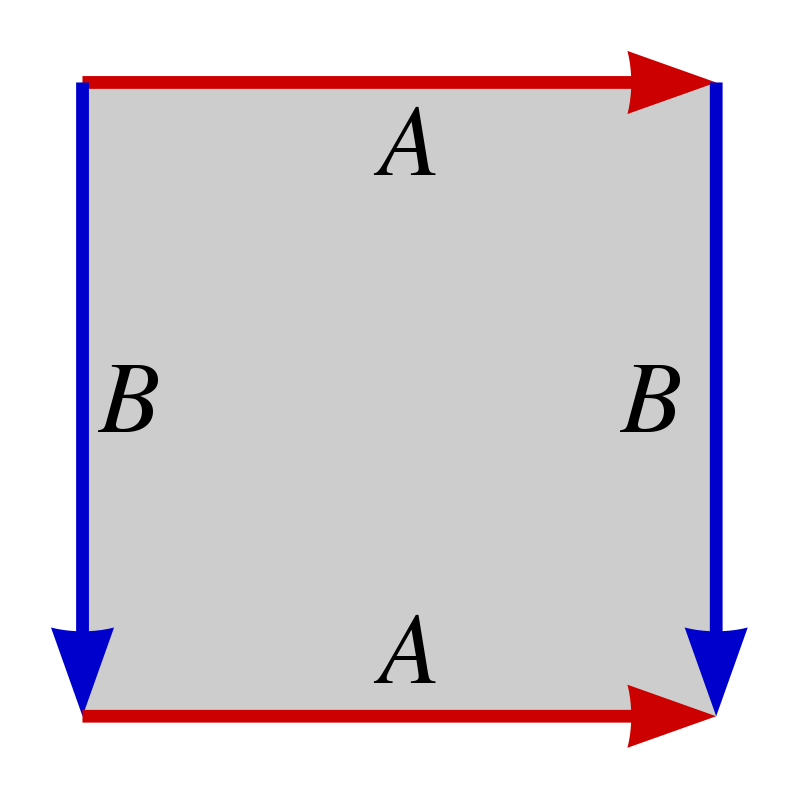}
    \caption{Polygone fondamental du tore de dimension 2}
    \label{fig:torus-as-square}
\end{figure}

\begin{definition}
On considère le carré $[-1,1]^2$, ainsi qu'une relation d'équivalence $\eqCT$ définie par : \[\forall(u,v),(u',v')\in [-1,1]^2,\quad (u,v)\eqCT(u',v')\Leftrightarrow\left\{\begin{matrix}
(u',v')=(u,\pm u) & \text{si }v=\pm1 \\ 
(u',v')=(\pm u,v) & \text{si }u=\pm1\\
(u',v')=(u,v)&\text{sinon}.
\end{matrix}\right.\]On définit ainsi l'espace quotient $C_T:=[-1,1]^2/\eqCT$.
\end{definition}

\subsection{Plongement dans l'espace euclidien}

\begin{figure}[H]
    \centering
    \includegraphics[width=0.5\linewidth]{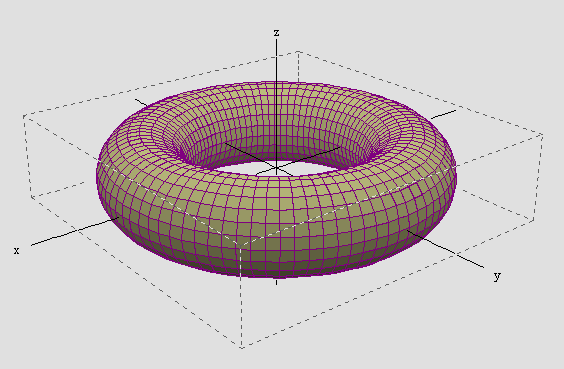}
    \caption{Surface paramétrée comme définition \ref{def:torus-R3}, avec $R=3r$}
    \label{fig:torus}
\end{figure}

De même que pour notre premier complexe cellulaire, nous pouvons plonger celui-ci dans l'espace $\bb{R}^3$, par la définition qui suit. De manière intuitive, nous pouvons imaginer partir d'une feuille de papier sur laquelle on recolle les côtés identiques du polygone fondamental, pour obtenir la figure ci-dessus.

\begin{definition}\label{def:torus-R3}
Soient $R>r>0$. On considère la surface paramétrée suivante : \begin{equation}\label{eq:torus}
\left\{\begin{matrix}
x(u,v)&=&\big(R+r\cos (v\pi)\big)\cos (u\pi)\\ 
y(u,v)&=&\big(R+r\cos (v\pi)\big)\sin (u\pi)\\ 
z(u,v)&=&r\sin (v\pi),
\end{matrix}\right.
\end{equation}avec $u,v\in[-1,1]$.

Nous définirons l'application $t:(u,v)\mapsto\big(x(u,v),y(u,v),z(u,v)\big)$, se référant aux notations de l'équation \eqref{eq:torus}, ainsi que la surface de $\bb{R}^3$ contenant tout les points de la paramétrisation $T=im(t)=\left\{t(u,v),u,v\in[-1,1]\right\}$.

Cette surface admet comme équation cartésienne la suivante :\[(\sqrt{x^2+y^2}-R)^2+z^2=r^2.\]
\end{definition}

\begin{remark}
La valeur de $v$ nous fait tourner autour du tube, tandis que la valeur de $u$ nous fait tourner le long de celui-ci.
\end{remark}

\begin{proposition}\label{prop:torus-relation}
L'application $t$ vérifie l'équivalence suivante : \begin{equation*}
\forall(u,v),(u',v')\in[-1,1]^2,\quad (u,v)\eqCT(u',v')\Leftrightarrow t(u,v)=t(u',v').
\end{equation*}
\end{proposition}
\begin{proof}
Montrons dans un premier temps l'implication $(\Rightarrow)$. Soient $(u,v),(u',v')$ deux couples de $[-1,1]^2$, tels que $(u,v)\eqCT(u',v')$. Montrons que l'on a l'égalité $t(u,v)=t(u',v')$.\begin{itemize}
    \item Si $u\neq 1$ ou $v\neq 1$, c'est évident, puisque $(u',v')=(u,v)$.
    \item Si $u=\pm1$, alors $(u',v')=(\pm u,v)$. De plus, on a : \[\begin{split}
t(-1,v)&=\left\{\begin{matrix}
x(-1,v)&=&\big(R+r\cos (v\pi)\big)\cos(-\pi)\\ 
y(-1,v)&=&\big(R+r\cos (v\pi)\big)\sin(-\pi)\\ 
z(-1,v)&=&r\sin (v\pi)
\end{matrix}\right.\\
&=\left\{\begin{matrix}
x(-1,v)&=&\big(R+r\cos (v\pi)\big)\cos \pi\\ 
y(-1,v)&=&\big(R+r\cos (v\pi)\big)\sin \pi\\ 
z(-1,v)&=&r\sin (v\pi)
\end{matrix}\right.\\
&=t(1,v).
\end{split}\]Dès lors, on peut en déduire que $t(u',v')=t( u,v)$.
\item Si $v=\pm1$, alors $(u',v')=(u,\pm v)$. De plus, on a : \[\begin{split}
t(u,-1)&=\left\{\begin{matrix}
x(u,-1)&=&\big(R+r\cos (-\pi)\big)\cos(u\pi)\\ 
y(u,-1)&=&\big(R+r\cos (-\pi)\big)\sin(u\pi)\\ 
z(u,-1)&=&r\sin (-\pi)
\end{matrix}\right.\\
&=\left\{\begin{matrix}
x(u,-1)&=&\big(R+r\cos \pi\big)\cos (u\pi)\\ 
y(u,-1)&=&\big(R+r\cos \pi\big)\sin (u\pi)\\ 
z(u,-1)&=&r\sin \pi
\end{matrix}\right.\\
&=t(u,1).
\end{split}\]Nous pouvons alors en déduire que $t(u',v')=t(u,v)$.
\end{itemize}
Nous pouvons alors conclure que $t$ est compatible avec la relation d'équivalence.

\bigskip Montrons désormais l'implication $(\Leftarrow)$. Soient $(u,v),(u',v')\in[-1,1]^2$ tels que $t(u,v)=t(u',v')$.\begin{itemize}
\item Supposons que $u=\pm\frac{1}{2}$. L'équation paramétrée nous donne : \[t\left(\pm\frac{1}{2},v\right)=t(u',v')\Rightarrow\left\{\begin{matrix}
0&=&(R+r\cos(v'\pi))\cos(u'\pi)\\ 
\pm\big(R+r\cos(v\pi)\big)&=&(R+r\cos(v'\pi))\sin(u'\pi)\\ 
r\sin(v\pi)&=&r\sin(v'\pi).
\end{matrix}\right.\]Avec les deux premières équations, on peut en déduire que $\cos(u'\pi)=0$. Ainsi, on sait que $u'\in\left\{-\frac{1}{2},\frac{1}{2}\right\}=\{\pm u\}$. Traitons alors chacun des cas : \begin{itemize}
    \item Si $u'=-u$, alors $\sin(u'\pi)=-\sin(u\pi)$, et ainsi la deuxième équation donnerait $(R+r\cos(v\pi))=-(R+r\cos(v'\pi))$. En regroupant les termes du même côté, on obtient $2R+r\big(\cos(v\pi)+\cos(v'\pi)\big)=0$, qui n'a pas de solution car $R>r$.
    \item Si $u'=u$, nous obtenons des deux dernières équations les égalités~$\sin(v\pi)=\sin(v'\pi)$ et $\cos(v\pi)=\cos(v'\pi)$. On en déduit soit $v'=v$, soit $v'=-v$ si $v=\pm1$, ce qui revient dans tout les cas à dire $(u',v')\eqCT(u,v)$.
\end{itemize}
\item Supposons maintenant que $u\neq\pm\frac{1}{2}$. Dans ce cas, nous avons $\cos(v\pi)\neq0$. Ainsi, on retrouve avec l'équation paramétrée :\[t(u,v)=t(u',v')\Rightarrow\left\{\begin{matrix}
(R+r\cos(v\pi))\cos(u\pi)&=&(R+r\cos(v'\pi))\cos(u'\pi)\\ 
(R+r\cos(v\pi))\sin(u\pi)&=&(R+r\cos(v'\pi))\sin(u'\pi)\\ 
r\sin(v\pi)&=&r\sin(v'\pi).
\end{matrix}\right.\]En divisant la deuxième équation par la première, on a $\tan(u\pi)=\tan(u'\pi)$. Puisque la tangente est $\pi$-périodique, on a $u'=u+\varepsilon$, avec $\varepsilon\in\{0,\pm1,\pm2\}$. Regardons les différents cas :\begin{itemize}
    \item Si $\varepsilon=0$, on a $u'=u$ et donc $\cos(u'\pi)=\cos(u\pi)$. Ainsi, avec simplification des deux premières équations, on obtient $\cos(v'\pi)=\cos(v\pi)$ et $\sin(v'\pi)=\sin(v\pi)$. On peut en déduire que $v'=v$, et donc par extension que $(u,v)\eqCT(u',v')$.
    \item Si $\varepsilon=\pm1$, on aurait $\cos(u'\pi)=-\cos(u\pi)$, et la première équation donnerait : \[\begin{split}
        &(R+r\cos(v\pi))\cos(u\pi)=-(R+r\cos(v'\pi))\cos(u\pi)\\
        \Longrightarrow\ &R+r\cos(v\pi)=-R-r\cos(v'\pi)\qquad \text{car }\cos(u\pi)\neq0\\
        \Longrightarrow\ &2R+r\cos(v\pi)=-r\cos(v'\pi).
    \end{split}\]Comme $\min\big(2R+\cos(v\pi)\big)>\max\big(-r\cos(v'\pi)\big)$, il n'existe pas de solution à cette équation. On en conclut que ces cas sont impossible.
    \item Si $\varepsilon=\pm2$, alors $u=\pm1$ et $u'=-u$. Étant donné que le cosinus et le sinus sont $2\pi$-périodique, on a $\cos(u'\pi)=\cos(u\pi)$ et $\sin(u'\pi)=\sin(u\pi)$. Dans ce cas, les deux premières équations donnent $\cos(v\pi)=\cos(v'\pi)$, et la dernière $\sin(v\pi)=\sin(v'\pi)$. Par la $2\pi$-périodicité des applications, nous avons $v'=v$ ou $v=1$ et $v'=-1$. Dans tout les cas, cela revient à dire que $(u',v')\eqCT(u,v)$.
\end{itemize}
\end{itemize}Nous pouvons alors en conclure que l'application $t$ vérifie $t(u,v)=t(u',v')$ si et seulement si $(u,v)\eqCT(u',v')$.
\end{proof}

Par la propriété universelle du quotient \ref{th:quotient}, nous savons qu'il existe une unique application injective $\overline{t}:C_T\to T$ tel que $t=\overline{t}\circ\pi_t$, avec $\pi_t$ la surjection canonique pour le quotient par $\eqCT$.

\begin{proposition}\label{prop:torus-reciproque}
L'application suivante est une inverse à droite de l'application $t$ : \[\begin{split}
g_t:T&\longrightarrow [-1,1]^2\\
(x,y,z)&\mapsto\left (\frac{1}{\pi}\atan2(y,x),\left\{\begin{matrix}
\frac{1}{\pi}\atan2(z,x-R)&\text{si } \atan2(y,x)=0\\
\frac{1}{\pi}\atan2(z,-x-R)&\text{si } \atan2(y,x)=\pi\\
\frac{1}{\pi}\atan2\left(z,\frac{y}{\sin\left(\atan2(y,x)\right)}-R\right)&\text{sinon}
\end{matrix}\right. \right ).
\end{split}\]
\end{proposition}
\begin{proof}
Nous voulons montrer l'égalité $t\circ g_t=id_T$. Soit $(x,y,z)\in~T$, pour lequel on choisit $(u,v)\in[-1,1]^2$ tel que $t(u,v)=(x,y,z)$. Posons enfin~$(u',v')=g_t(x,y,z)$. On veut alors montrer que $(u,v)\eqCT(u',v')$, ce qui suffit car $t\circ g_t(x,y,z)=t(u',v')=t(u,v)=(x,y,z)$. \begin{itemize}
    \item Si $u=0$, alors $t(0,v)=(R+r\cos(v\pi),0,r\sin(v\pi))$. Comme on a~$y=0$ et~$x>0$, on obtient :\begin{equation*}\label{eq:torus-proof-inv}
g_t(x,y,z)=g_t(R+r\cos(v\pi),0,r\sin(v\pi))=\left(0,\frac{1}{\pi}\atan2(r\sin(v\pi),r\cos(v\pi))\right).
    \end{equation*}On remarque ici que $u'=0=u$. Traitons ainsi les cas suivants séparément.\begin{itemize}
        \item Si $v=1$, alors $\frac{1}{\pi}\atan2(r\sin(v\pi),r\cos(v\pi))=\frac{1}{\pi}\atan2(0,-1)=1$. Ce qui veut dire que $v=v'$, donc $(u,v)=(u',v')$
        \item Si $v=-1$, alors $\frac{1}{\pi}\atan2(r\sin(v\pi),r\cos(v\pi))=\frac{1}{\pi}\atan2(0,-1)=1$. Or, comme $(0,1)\eqCT(0,-1)$, on a donc $(u,v)\eqCT(u',v')$.
        \item Si $v\neq\pm1$, alors $\atan2$ retourne l'angle $v\pi$. Autrement dit, on a bien~$v=v'$, donc $(u,v)\eqCT(u',v')$
    \end{itemize}
    \item Si $u=\frac{1}{2}$, alors $t\big(\frac{1}{2},v\big)=\big(0,R+r\cos(v\pi),r\sin(v\pi)\big)$. Or, comme on a~$\atan2(y,x)=\atan2(R+r\cos(v\pi),0)=\frac{\pi}{2}$, on obtient : \[g_t(x,y,z)=g_t\big(0,R+\cos(v\pi),r\sin(v\pi)\big)=\left(\frac{1}{2},\frac{1}{\pi}\atan2(r\sin(v\pi),r\cos(v\pi)\right).\]Cela nous ramène au résultat précédent pour la composante de droite, de plus que $u'=\frac{1}{2}=u$. Nous pouvons alors en déduire $\big(u,v)\eqCT(u',v')$.
    \item Si $u=-\frac{1}{2}$, on a $t\big(-\frac{1}{2},v\big)=\big(0,-R-r\cos(v\pi),r\sin(v\pi)\big)$. Dans ce cas, on a $y<0$, donc $g_t(x,y,z)=\left(-\frac{1}{2},\frac{1}{\pi}\atan2(r\sin(v\pi),r\cos(v\pi)\right)$. Ici encore, on peut en déduire que $\big(u,v\big)\eqCT(u',v')$.
    \item Si $u=\pm1$, on a $(x,y,z)=t(\pm1,v)=\big(-(R+r\cos(v\pi)), 0, r\sin(v\pi)\big)$. En remarquant que $\atan2(y,x)=\atan2(0,-R-r\cos(v\pi))=\pi$, on en déduit que~$g_t(x,y,z)=\left(1,\atan2(z,R-x)\right)=\big(1,\atan2(r\sin(v\pi),r\cos(v\pi))\big)$. On a $u'=\pm u$ et $v'=v$, c'est à dire $(u,v)\eqCT(u',v')$.
    \item Enfin si $u\not\in\{0,\pm\frac{1}{2},\pm1\}$, alors $x\neq0$. On peut calculer $\frac{y}{x}=\tan(u\pi)$, ce qui nous permet directement d'en déduire que $u'\pi=\atan2(y,x)=u\pi$.

    Intéressons-nous à la seconde composante de $g_t$. On a : \[\frac{y}{\sin(\atan2(y,x))}-R=\frac{\big(R+r\cos(v\pi)\big)\sin(u\pi)}{\sin(u\pi)}-R=r\cos(v\pi),\]dont on peut en déduire le résultat suivant : $$g_t(x,y,z)=(u,\atan2\big(r\sin(v\pi),r\cos(v\pi))\big)=(u,v).$$Nous pouvons alors en déduire que $(u,v)=(u',v')$.
\end{itemize}
Nous pouvons alors en conclure que $g_t$ est bien une inverse à droite de l'application $t$.
\end{proof}

\begin{lemma}\label{lemma:torus-inv-continuous}
L'application $\pi_t\circ g_t$ est l'inverse de $\overline{t}$, et elle est continue.
\end{lemma}
\begin{proof}
On peut, dans un premier temps, voir que l'on a : \[\overline{t}\circ\pi_t\circ g_t=t\circ g_t=id_T.\]Cela montre que l'application $\overline{t}$ admet $\pi_t\circ g_t$ comme inverse à droite. 

\bigskip Montrons désormais que l'application $\pi_t\circ g_t$ est continue. Tout d'abord, comme pour le lemme \ref{lemma:mobius-inv-continuous}, la première composante de $g_t$ devient continue avec la composition avec $\pi_t$.

Montrons la continuité de la deuxième composante. Si $\atan2(y,x)\neq 0$, cette composante est donnée par $\atan2\left(z,\frac{y}{sin(\atan2(y,x))}-R\right)$. De part ce que l'on vient de voir, cette application est bien continue lorsque l'on compose avec la surjection $\pi_t$. Il nous reste alors à montrer que cette application tende bien vers~$\atan2(z,x-R)$ lorsque $\atan2(y,x)$ tend vers 0, et vers $\atan2(z,-x-R)$ lorsque $\atan2(y,x)$ tend vers $\pi$. Autrement dit, il nous suffit de montrer que lorsque $y$ tend vers 0, $\frac{y}{\sin(\atan2(y,x))}$ tend vers $|x|$.

\bigskip Comme nous nous intéressons à la limite lorsque $y$ tend vers 0, on peut alors supposer $x$ non nul. Dans ce cas, lorsque $x>0$, atan2 est égal à l'arctan, et lorsque $x<0$, il faut lui rajouter une rotation de $\pm\pi$. Dans ce cas, nous pouvons remarquer que l'on a :\[\begin{split}
\sin(\atan2(y,x))&=\sin\left(\arctan\left(\frac{y}{x}\right)\pm\pi\right)\\
&=\sin\left(\arctan\left(\frac{y}{x}\right)\right)\cos(\pi)\pm\cos\left(\arctan\left(\frac{y}{x}\right)\right)\sin(\pi)\\
&=-\sin\left(\arctan\left(\frac{y}{x}\right)\right)\\
&=\sin\left(\arctan\left(-\frac{y}{x}\right)\right)\\
&=\sin\left(\arctan\left(\frac{y}{|x|}\right)\right).
\end{split}\]On peut alors écrire que $\sin(\atan2(y,x))=\sin(\arctan(\frac{y}{|x|}))$ pour $x\neq0$. Par la suite, en utilisant la formule de la proposition \ref{prop:sin-arc}, on peut obtenir les égalités suivantes :\[\begin{split}
\sin(\atan2(y,x))&=\sin\left(\arctan\left(\frac{y}{|x|}\right)\right)\\
&=\frac{y}{|x|}\frac{1}{\sqrt{1+(\frac{y}{|x|})^2}}\\
&=\frac{y}{|x|}\frac{|x|}{\sqrt{x^2+y^2}}\\
&=\frac{y}{\sqrt{x^2+y^2}}.
\end{split}\]En utilisant le développement limité de $r=\sqrt{x^2+y^2}$, qui vaut $|x|+\frac{y^2}{|x|}+o(y^3)$, on peut alors en déduire : \[\begin{split}
\frac{y}{\sin(\atan2(y,x))}&=\frac{y}{\frac{y}{\sqrt{x^2+y^2}}}+o(y^3)\\
&=\sqrt{x^2+y^2}+o(y^3)\\
&=|x|+\frac{y^2}{|x|}+o(y^3).
\end{split}\] 

On peut alors en conclure que la limite de~$\frac{y}{\sin(\atan2(y,x))}$ vaut bien $|x|$ lorsque~$y$ tend vers 0.

Nous pouvons ainsi conclure que l'application $\pi_t\circ g_t$ est continue sur $T$.
\end{proof}

\begin{theorem}\label{th:torus-homeo}
L'espace quotient $C_T$ et le sous espace $T$ de $\bb{R}^3$ sont homéomorphes.
\end{theorem}
\begin{proof}
En utilisant la proposition \ref{prop:torus-relation} ainsi que la propriété universelle quotient, nous pouvons dire que l'application $\overline{t}$ est injective et continue. Dans le lemme \ref{lemma:torus-inv-continuous}, nous avons montré que $\pi_t\circ g_t$ était l'inverse à droite de $\overline{m}$, et qu'elle était elle aussi continue. Finalement, avec la proposition \ref{prop:inj+inv=bij}, nous pouvons en déduire que $\overline{t}$ est une application bijective et continue. Son inverse, étant $\pi_t\circ g_t$, est également continue, ce qui nous permet de dire que $\overline{t}$ est un homéomorphisme entre l'espace quotient $C_T$ et la surface $T$ paramétrée de $\bb{R}^3$.
\end{proof}

\section{Le plan projectif réel}

\subsection{Polygone fondamental}

Le dernier espace est celui que l'on appelle le \emph{plan projectif réel}, que l'on appellera par la suite simplement plan projectif. C'est un complexe cellulaire sur lequel on identifie chaque côté d'un carré par son opposé dans le sens inverse, comme sur la figure suivante.

\begin{figure}[H]
    \centering
    \includegraphics[width=0.27\linewidth]{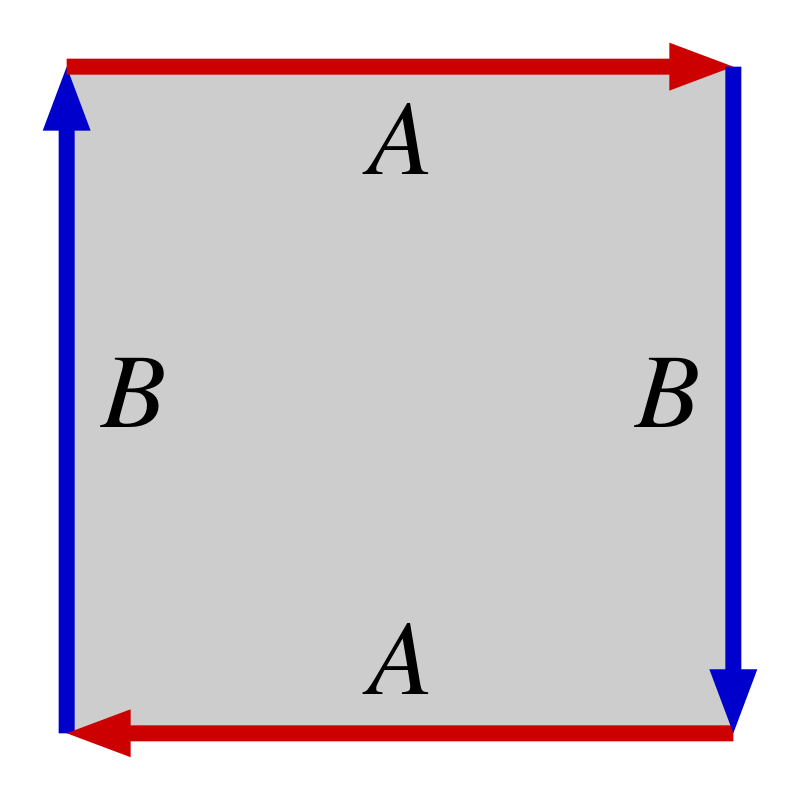}
    \caption{Polygone fondamental du plan projectif}
    \label{fig:P2R-as-square}
\end{figure}

\begin{definition}
On considère le carré $[-1,1]^2$, ainsi qu'une relation d'équivalence  définie sur celui-ci, par : \[\forall(x,y),(x',y')\in [-1,1]^2,\quad (x,y)\eqCP(x',y')\Leftrightarrow \left\{\begin{matrix}
(x',y')=\pm(x,y)&\text{si }x=\pm1\\ 
(x',y')=\pm(x,y)&\text{si }y=\pm1\\ 
(x',y')=(x,y)&\text{sinon}.
\end{matrix}\right.\]On définit ainsi l'espace quotient $C_P:=[-1,1]^2/\eqCP$.
\end{definition}

\subsubsection{Homéomorphisme à l'hémisphère}

On considère la sphère $\s{2}$, que nous allons quotienter par la relation de symétrie par rapport à l'origine. Nous obtenons alors l'hémisphère $S=\s{2}/\eqSym$. Cet ensemble peut être perçu comme une demi-sphère (\ref{fig:sphere-P2R}), sur laquelle on identifie les points antipodaux du cercle. Nous noterons $[x,y,z]$ un élément de $S$, qui est la classe d'équivalence de $(x,y,z)\in \s{2}$.
\begin{figure}[H]
    \centering
    \includegraphics[width=0.2\linewidth]{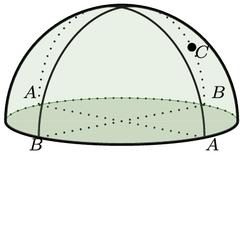}
    \caption{Demi-sphère $S=\s{2}/\eqSym$}
    \label{fig:sphere-P2R}
\end{figure}

\begin{proposition}\label{homeo-S-CP}
Les espaces quotients $C_P$ et $S$ sont homéomorphes.
\end{proposition}
\begin{proof}
Montrons qu'il existe deux applications réciproques l'une de l'autre, permettant de passer d'un espace à un autre.

Dans un premier temps, on s'intéresse à l'application définie par : \[\begin{split}
    f:[-1,1]^2&\longrightarrow S\\
    (x,y)&\mapsto\left\{\begin{matrix}
    [0,0,1]&\text{si } (x,y)=(0,0)\\
\left[\frac{\left \|(x,y)  \right \|_\infty}{\left \| (x,y) \right \|_2}x,\frac{\left \|(x,y)  \right \|_\infty}{\left \| (x,y) \right \|_2}y,\sqrt{1-\left \|(x,y)  \right \|_\infty^2}\right]&\text{sinon.}
\end{matrix}\right.
\end{split}\]

Montrons que l'application est bien définie. Soit $(x,y)\in[-1,1]^2$. On a : \[\left(\frac{\left \|(x,y)  \right \|_\infty}{\left \| (x,y) \right \|_2}x\right)^2+\left(\frac{\left \|(x,y)  \right \|_\infty}{\left \| (x,y) \right \|_2}y\right)^2=\left \| (x,y) \right \|_\infty^2\frac{x^2+y^2}{\sqrt{x^2+y^2}^2}=\left \| (x,y) \right \|_\infty^2.\]On peut alors en déduire que le point $f(x,y)$ appartient à la sphère $\s{2}$, étant solution de l'équation cartésienne de celle-ci.

De plus, on a, pour $y\in[-1,1]$ :\[\begin{split}
f(-1,y)&=\left[\frac{-1}{\sqrt{1+y^2}},\frac{y}{\sqrt{1+y^2}},\sqrt{1-1}\right]\\
&=\left[\frac{-1}{\sqrt{1+y^2}},\frac{y}{\sqrt{1+y^2}},0\right]\\
&=\left[\frac{1}{\sqrt{1+y^2}},\frac{-y}{\sqrt{1+y^2}},0\right]\quad\text{ avec la relation de symétrie sur }S\\
&=f(1,-y).
\end{split}\]De la même manière pour $x\in[-1,1]$, on a $f(x,-1)=f(-x,1)$, ce qui nous permet de conclure que l'application est compatible avec la relation $\eqCP$. De plus, l'application $f$ est continue sur $[-1,1]^2\setminus\{(0,0)\}$, par produit et combinaisons d'applications continues. Il reste à montrer la continuité en $(0,0)$. Il est évident que la limite de la troisième composante de $f$ lorsque $(x,y)$ tend vers $(0,0)$ est bien 1. Montrons maintenant que $\lim_{(x,y)\rightarrow(0,0)}\frac{||(x,y)||_\infty}{||(x,y)||_2}x=0$. En supposant~$||(x,y)||_\infty=|x|$, nous avons le résultat que suit : \[\left|\frac{||(x,y)||_\infty}{||(x,y)||_2}x\right|=\frac{x^2}{\sqrt{x^2+y^2}}\leq\frac{x^2}{|x|}\underset{(x,y)\rightarrow(0,0)}{\longrightarrow}0.\]

En supposant cette fois $||(x,y)||_\infty=y$, nous avons le résultat suivant : \[\left|\frac{||(x,y)||_\infty}{||(x,y)||_2}x\right|=\frac{|xy|}{\sqrt{x^2+y^2}}\leq\frac{|xy|}{|x|}\underset{(x,y)\rightarrow(0,0)}{\longrightarrow}0.\]Nous pouvons alors en conclure que le terme $\frac{||(x,y)||_\infty}{||(x,y)||_2}x$ tend vers 0 lorsque $(x,y)$ tend vers $(0,0)$. La limite de $\frac{||(x,y)||_\infty}{||(x,y)||_2}y$ se traite de la même manière, et tend également vers 0 lorsque $(x,y)$ tend vers $(0,0)$. Nous pouvons alors en conclure que l'application $f$ est continue en $(0,0)$. Ainsi, l'application $f$ est bien continue.

\bigskip En utilisant la propriété universelle du quotient \ref{th:quotient}, on en déduit qu'il existe une unique application continue $\overline{f}$ définie par : \[\begin{split}
\overline{f}:C_P&\longrightarrow S\\
[x,y]&\mapsto\left\{\begin{matrix}
    [0,0,1]&\text{si } (x,y)=(0,0)\\
\left[\frac{\left \|(x,y)  \right \|_\infty}{\left \| (x,y) \right \|_2}x,\frac{\left \|(x,y)  \right \|_\infty}{\left \| (x,y) \right \|_2}y,\sqrt{1-\left \|(x,y)  \right \|_\infty^2}\right]&\text{sinon.}
\end{matrix}\right.
\end{split}\]Nous voulons trouver une application continue inverse à celle-ci, permettant de conclure que les espaces sont bien homéomorphes.
Dans un second temps, on s'intéresse à l'application définie par : \[\begin{split}
g:\s{2}&\longrightarrow C_P\\
(x,y,z)&\mapsto\left\{\begin{matrix}
    [0,0]&\text{si } (x,y,z)=(0,0,1)\\
\left[\varepsilon(z)\frac{\left \| (x,y) \right \|_2}{\left \| (x,y) \right \|_\infty}x,\varepsilon(z)\frac{\left \| (x,y) \right \|_2}{\left \| (x,y) \right \|_\infty}y\right]&\text{sinon.}
\end{matrix}\right.
\end{split}\]avec $\varepsilon:x\mapsto\pm1,  \varepsilon(0)=1$ l'application qui retourne le signe de $x$. Montrons que cette application est compatible avec la relation de symétrie. Tout d'abord, nous allons traiter séparément le cas où $z=0$. 

Soit $(x,y,0)\in S$. Dans cette configuration, on a $||(x,y)||_2=x^2+y^2=1$, ce qui nous donne : \[g(x,y,0)=\left[\frac{x}{\left \| (x,y) \right \|_\infty},\frac{y}{\left \| (x,y) \right \|_\infty}\right].\]Or, l'image du point antipodal nous donne : \[g(-x,-y,0)=\left[\frac{-x}{\left \| (-x,-y) \right \|_\infty},\frac{-y}{\left \| (-x,-y) \right \|_\infty}\right]=\left[\frac{-x}{\left \| (x,y) \right \|_\infty},\frac{-y}{\left \| (x,y) \right \|_\infty}\right].\] Du fait que $||(x,y)||_\infty$ vaut $|x|$ ou $|y|$, on aura un terme qui vaudra $\pm1$, appartenant donc au bord du carré $C_P$. On en déduit que $g(-x,-y,0)=g(x,y,0)$, du fait que $\left(\frac{x}{\left \| (x,y) \right \|_\infty},\frac{y}{\left \| (x,y) \right \|_\infty}\right)\eqSym\left(\frac{-x}{\left \| (x,y) \right \|_\infty},\frac{-y}{\left \| (x,y) \right \|_\infty}\right)$.

Traitons désormais le cas général. Soit $(x,y,z)\in S$, on a :\[\begin{split}
g(-x,-y,-z)&=\left[\varepsilon(-z)\frac{\left \| (-x,-y) \right \|_2}{\left \| (-x,-y) \right \|_\infty}(-x),\varepsilon(-z)\frac{\left \| (-x,-y) \right \|_2}{\left \| (-x,-y) \right \|_\infty}(-y)\right]\\
&=\left[-\varepsilon(z)\frac{\left \| (x,y) \right \|_2}{\left \| (x,y) \right \|_\infty}(-x),-\varepsilon(z)\frac{\left \| (x,y) \right \|_2}{\left \| (x,y) \right \|_\infty}(-y)\right]\\
&=\left[\varepsilon(z)\frac{\left \| (x,y) \right \|_2}{\left \| (x,y) \right \|_\infty}x,\varepsilon(z)\frac{\left \| (x,y) \right \|_2}{\left \| (x,y) \right \|_\infty}y\right]\\
&=g(x,y,z).
\end{split}\]

Montrons que cette application est continue. Il est clair qu'elle est continue sur $\s{2}\setminus\{(0,0,1)\}$. Il reste à montrer qu'elle l'est également en $(0,0,1)$. Nous allons étudier la limite de $\varepsilon(z)\frac{\left \| (x,y) \right \|_2}{\left \| (x,y) \right \|_\infty}x$ lorsque $(x,y,z)$ tend vers $(0,0,1)$.

Supposons que $||(x,y)||_\infty=|x|$. Dans ce cas, nous avons le résultat suivant : \[\left|\varepsilon(z)\frac{\left \| (x,y) \right \|_2}{\left \| (x,y) \right \|_\infty}x\right|=\frac{\sqrt{x^2+y^2}}{|x|}|x|=\sqrt{x^2+y^2}\underset{(x,y,z)\to(0,0,1)}{\longrightarrow}0.\]Supposons cette fois que $||(x,y)||_\infty=|y|$. Dans ce cas, nous avons le résultat suivant : \[\left|\varepsilon(z)\frac{\left \| (x,y) \right \|_2}{\left \| (x,y) \right \|_\infty}x\right|=\frac{\sqrt{x^2+y^2}}{|y|}|x|\leq\frac{\sqrt{x^2+y^2}}{|y|}|y|=\sqrt{x^2+y^2}\underset{(x,y,z)\to(0,0,1)}{\longrightarrow}0.\]De manière analogue, on peut montrer que la deuxième composante $\varepsilon(z)\frac{\left \| (x,y) \right \|_2}{\left \| (x,y) \right \|_\infty}y$ tend vers 0 lorsque $(x,y,z)$ tend vers $(0,0,1)$. On peut alors en conclure que l'application $g$ est continue.

\bigskip L'application $g$ est donc continue et compatible avec la relation de symétrie, qui nous permet d'affirmer l'existence d'une unique application continue $\overline{g}$, définie par : \[\begin{split}
\overline{g}:S&\longrightarrow C_P\\
[x,y,z]&\mapsto\left\{\begin{matrix}
    [0,0]&\text{si } (x,y,z)=(0,0,1)\\
\left[\varepsilon(z)\frac{\left \| (x,y) \right \|_2}{\left \| (x,y) \right \|_\infty}x,\varepsilon(z)\frac{\left \| (x,y) \right \|_2}{\left \| (x,y) \right \|_\infty}y\right]&\text{sinon.}
\end{matrix}\right.
\end{split}\]

Enfin, montrons que les applications $\overline{f}$ et $\overline{g}$ sont réciproques l'une de l'autre. Soit~$[x,y]\in C_P$, on a : \[\begin{split}
\overline{g}\circ\overline{f}\big([x,y]\big)&=\left[\varepsilon(z)\frac{\left \| (x,y) \right \|_2}{\left \| (x,y) \right \|_\infty}\frac{\left \| (x,y) \right \|_\infty}{\left \| (x,y) \right \|_2}x,\varepsilon(z)\frac{\left \| (x,y) \right \|_2}{\left \| (x,y) \right \|_\infty}\frac{\left \| (x,y) \right \|_\infty}{\left \| (x,y) \right \|_2}y\right]\\
&=\big[\varepsilon(z)x,\varepsilon(z)y]\\
&=\big[x,y\big].
\end{split}\]avec $z=\sqrt{1-\left \|(x,y)  \right \|_\infty^2}$.

Soit $(x,y,z)\in S$, on a : \[\begin{split}
\overline{f}\circ\overline{g}\big([x,y,z]\big)&=\left[\varepsilon(z)x,\varepsilon(z)y,\sqrt{1-\left \| \left( \varepsilon(z)\frac{\left \| (x,y) \right \|_2}{\left \| (x,y) \right \|_\infty}x,\varepsilon(z)\frac{\left \| (x,y) \right \|_2}{\left \| (x,y) \right \|_\infty}y\right)\right\|_\infty^2}\right]\\
&=\left[\varepsilon(z)x,\varepsilon(z)y,\sqrt{1-\varepsilon(z)^2\frac{\left \| (x,y) \right \|_2^2}{\left \| (x,y) \right \|_\infty^2}\left \|(x,y)  \right \|_\infty^2}\right]\\
&=\left[\varepsilon(z)x,\varepsilon(z)y,\sqrt{1-\left \| (x,y) \right \|_2^2}\right]\\
&=[\varepsilon(z)x,\varepsilon(z)y,|z|]\\
&=[\varepsilon(z)x,\varepsilon(z)y,\varepsilon(z)z]\\
&=[x,y,z].
\end{split}\]Dès lors, nous pouvons en conclure que les applications $\overline{f}$ et $\overline{g}$ sont réciproques l'une de l'autre et continues, ce qui fait que $S$ et $C_P$ sont homéomorphes.
\end{proof}

\subsection{Plongement dans l'espace euclidien}

L'objectif de cette partie est de montrer que le plan projectif peut être représenté en tant que sous-espace dans l'espace euclidien à quatre dimensions.

\begin{definition}
On considère la surface paramétrée suivante à partir d'éléments de la sphère $\s{2}$ : \begin{equation}
\left\{\begin{matrix}
u(x,y,z)&=&xy\\ 
v(x,y,z)&=&xz\\ 
w(x,y,z)&=&y^2-z^2\\ 
t(x,y,z)&=&2yz.
\end{matrix}\right.
\end{equation}On note $p:(x,y,z)\in\s{2}\rightarrow \bb{R}^4$ l'application continue paramétrant la surface. On notera le sous-espace contenant tout les points comme suit : $$P:=im(p)=\{u(x,y,z), v(x,y,z), w(x,y,z), t(x,y,z),(x,y,z)\in~\s{2}\}.$$
\end{definition}

\begin{proposition}\label{prop:projective-relation}
L'application $p$ vérifie l'équivalence suivante : \begin{equation*}
\forall(x,y,z),(x',y',z')\in\s{2},\quad (x,y,z)\eqSym(x',y',z')\Leftrightarrow p(x,y,z)=p(x',y',z').
\end{equation*}
\end{proposition}
\begin{proof}
Dans un premier temps, montrons l'implication $(\Rightarrow)$. Les mônomes étant tout quadratiques, le signe des termes ne change rien. Alors nous pouvons directement en déduire que pour tout élément $(x,y,z)$ de la sphère $\s{2}$, on a $p(-x,-y,-z)=p(x,y,z)$.

\bigskip Montrons désormais la réciproque $(\Leftarrow)$. Soient $(x,y,z),(x',y',z')\in \s{2}$, tels que $p(x,y,z)=p(x',y',z')$. Montrons alors que $(x,y,z)\eqSym(x',y',z')$. On a le système d'équations suivant : \begin{equation*}
\left\{\begin{matrix}
xy=x'y'\\ 
xz=x'z'\\ 
y^2-z^2=y'^2-z'^2\\ 
2yz=2y'z'.
\end{matrix}\right.
\end{equation*}
Montrons d'abord par l'absurde que l'un des termes $x,y$ ou $z$ vaut 0 si et seulement si le terme correspondant $x',y'$ ou $z'$ vaut également 0.\begin{itemize}
    \item Supposons $x=0$ et $x'\neq0$. Alors, on a $x'y'=0$ et $x'z'=0$, dont on en déduit $y'=0$ et $z'=0$. Ainsi, on en déduit par la dernière équation que~$y=0$ ou $z=0$, et par la troisième que $y=z=0$. Finalement, on aboutit à $(x,y,z)=(0,0,0)$, un point n'appartenant pas à la sphère. Nous avons une contradiction. De manière symétrique, on montre que $x'=0$ et~$x\neq0$ est impossible.

    On en déduit que $x=0\Leftrightarrow x'=0$.
    \item Supposons $y=0$ et $y'\neq0$. La première et la dernière équations implique que $x'=z'=0$. Dès lors, on obtient par la troisième équation~$-z^2=y'^2$, ce qui est absurde puisque $-z^2$ ne peut être strictement positif. De manière symétrique, on démontre que $y'=0$ et $y\neq0$ est impossible. Puis de manière analogue, on peut démontrer de même pour $z$.

    Finalement, on obtient $y=0\Leftrightarrow y'=0$ et $z=0\Leftrightarrow z'=0$.
\end{itemize}

Ensuite, nous allons traiter le cas où deux des trois composante est nulle. Supposons que c'est $x$ qui est non nulle. Alors, par ce qui précède, on a~$x'\neq0$. L'équation de la sphère nous permet de dire que~$x^2=1=x'^2$, ce qui veut dire que $x'=\pm x$. Autrement dit, on a~$(x',0,0)=\pm(x,0,0)$, ce qui revient à dire $(x,0,0)\eqSym(x',0,0)$. Les deux autres cas se traite de la même manière.

\bigskip Désormais, nous allons traiter le cas où l'une des composante est nulle : \begin{itemize}
    \item Supposons $x=0$, ainsi que $y\neq0, z\neq0$. Cela veut dire aussi que $x'=0$, et $y'\neq0,z'\neq0$. Il nous reste alors les équations $y^2-z^2=y'^2-z'^2$ et~$yz=y'z'$, ainsi que celle de la sphère $y^2+z^2=1=y'^2+z'^2$. En sommant les deux équations faisant intervenir un carré, nous obtenons~$2y^2=2y'^2$, c'est à dire $y'=\pm y$. Si $y'=y$, alors $yz=y'z'=yz'$, ce qui veut dire que~$z=~z'$. Autrement dit, on a $(0,y',z')=(0,y,z)$. Dans le cas où~$y'=~-y$, alors~$yz=y'z'=-yz'$, c'est à dire $z'=-z$. Ainsi, on a $(0,y',z')=-(0,y,z)$.

    Dans tout les cas, on a $(0,y,z)\eqSym(0,y',z')$.
    \item Supposons $y=0$, et $x\neq0,z\neq0$. Comme précédemment, cela veut dire que $y'=0$, ainsi que $x'\neq0,z'\neq0$. Nous avons alors les équations~$xz=~x'z'$ et $z^2=z'^2$. Cette dernière nous affirme que~$z'=~\pm z$. Avec le même raisonnement que le point ci-dessus, nous avons les implications $z'=-z\Rightarrow x'=-x$ et $z'=z\Rightarrow x'=x$. On en conclut directement que $(x,0,z)\eqSym(x',0,z')$.
    \item Le cas $z=0$ et $x\neq0,y\neq0$ se traite comme le précédent, nous pouvons alors directement conclure que $(x,y,0)\eqSym(x',y',0)$.
\end{itemize}

Il ne reste plus qu'à traiter le cas où aucun terme est nul. Dans ce cas, nous obtenons les équations suivantes : \[\left\{\begin{matrix}
\frac{x}{x'}=\frac{y'}{y}\\ 
\frac{x}{x'}=\frac{z'}{z}\\ 
y^2-z^2=y'^2-z'^2\\ 
\frac{y}{y'}=\frac{z'}{z}.\end{matrix}\right.\]En mettant les équations ensemble, nous obtenons : \[\frac{x}{x'}=\frac{y'}{y}=\frac{z'}{z}=\frac{y}{y'}=\frac{x'}{x}=\frac{z}{z'}.\]Les points $(x,y,z)$ et $(x',y',z')$ appartiennent à la même droite vectorielle d'équation paramétrique $\frac{x''}{x'}=\frac{y''}{y'}=\frac{z''}{z'}$, avec $(x'',y'',z'')$ les variables de l'espace. En effet, ces équations sont vraies pour $(x'',y'',z'')=(x,y,z)$ d'après ce que l'on vient de voir, et vraie pour $(x'',y'',z'')=(x',y',z')$ car tous les facteurs valent 1.

Puisque les points appartient à la même droite vectorielle, ils sont soient égaux, soient antipodaux, ce qui veut dire que $(x,y,z)\eqSym(x',y,z')$.
\end{proof}

Par la propriété universelle du quotient \ref{th:quotient}, on sait alors qu'il existe une unique application injective continue $\overline{p}:\s{2}/\eqSym\to\bb{R}^4$, 
telle que $\overline{p}\circ\pi_p=p$, avec $\pi_p$ la surjection canonique pour le quotient par $\eqSym$.

\bigskip Notre objectif à présent est de montrer qu'il existe une application  inverse à droite de~$\overline{p}$, par recollement d'applications définies partiellement. Pour ce faire, nous allons d'abord montrer que de telles applications peuvent exister.

\begin{lemma}\label{lemma:projective-square}
Soit $(u,v,w,t)\in P$, avec $(x,y,z)\in\s{2}$ tel que~$p(x,y,z)=~(u,v,w,t)$. Alors:\[\begin{split}
    (x^2,y^2,z^2)&=\left(1-w-\frac{vt}{u}, \frac{vt}{2u}+w, \frac{vt}{2u}\right)\quad\text{si }u\neq0\\
    &=\left(1+w-\frac{ut}{v},\frac{ut}{2v},\frac{ut}{2v}-w\right)\quad\text{si }v\neq0\\
    &=\left(\frac{u^2-v^2}{w},\frac{1}{2}\left(1+w-\frac{u^2-v^2}{w}\right),\frac{1}{2}\left(1-w-\frac{u^2-v^2}{w}\right)\right)\quad\text{si }w\neq0\\
    &=\left(\frac{2uv}{t},\frac{1}{2}\left(1+w-\frac{2uv}{t}\right),\frac{1}{2}\left(1-w-\frac{2uv}{t}\right)\right)\quad\text{si }t\neq0.
\end{split}\]
\end{lemma}
\begin{proof}
Nous allons traiter chaque cas :
\begin{itemize}
    \item Supposons $u\neq0$.  On peut remarquer que $vt=xz\cdot 2yz=2uz^2$. Dès lors, on obtient bien $z^2=\frac{vt}{2u}$. De plus, on sait que $w=y^2-z^2$, donc on peut en déduire avec le résultat précédent que $y^2=\frac{vt}{2u}+w$. Enfin, avec l'équation de la sphère $x^2+y^2+z^2=1$, on en déduit donc la valeur de $x^2$ : \[x^2=1-y^2-z^2=1-\frac{vt}{2u}-w-\frac{vt}{2u}=1-w-\frac{vt}{u}.\]
    \item Supposons $v\neq0$. On peut remarquer que l'on a $ut=xy\cdot 2yz=2vy^2$. Dès lors, on obtient bien $y^2=\frac{ut}{2v}$. Pour la même raison que pour le point d'avant, on a~$z^2=y^2-w=\frac{ut}{2v}-w$. Et enfin, de même, on a~$x^2=1+w-\frac{ut}{v}$.
    \item Supposons $w\neq 0$. On remarque que l'on a~$u^2-v^2=x^2(y^2-z^2)=x^2w$. Dès lors, on obtient bien $x^2=\frac{u^2-v^2}{w}$. Ensuite, avec l'égalité $y^2=1-x^2-z^2$, si l'on ajoute $y^2$, on obtient : \[2y^2=1-x^2+y^2-z^2=1-\frac{u^2-v^2}{w}+w.\]Nous pouvons alors en déduire que l'on a $y^2=\frac{1}{2}\left(1+w-\frac{u^2-v^2}{w}\right)$. Enfin, avec le même raisonnement, on a : \[2z^2=1-x^2-y^2+z^2=1-w-\frac{u^2-v^2}{w}.\]Nous pouvons alors en déduire que l'on a bien $z^2=\frac{1}{2}\left(1-w-\frac{u^2-v^2}{w}\right)$.
    \item Supposons enfin $t\neq 0$. En remarquant $2uv=2xy\cdot xz=x^2t$, on obtient bien $x^2=\frac{2uv}{t}$. Dès lors, nous pouvons en déduire les valeurs de~$y^2$ et de $z^2$ de la même manière que précédemment, en changeant la valeur de~$x$ correspondante. On obtient donc bien $y^2=\frac{1}{2}\left(1+w-\frac{2uv}{t}\right)$, et également~$z^2=\frac{1}{2}\left(1-w-\frac{2uv}{t}\right)$.
\end{itemize}
\end{proof}

De cette manière, nous pouvons définir les applications suivantes. Le domaine de définition de $g_0$ est $\{(0,0,0,0)\}$ ; et celui de $g_i$, pour $i\in\{u,v,w,t\}$, est $P$ privé des points où $i=0$.
\[\left\{\begin{matrix}
g_0:(0,0,0,0)\mapsto(1,0,0)\\
g_u:(u,v,w,t)\mapsto\left(\sqrt{1-w-\frac{vt}{u}} ,\varepsilon(u)\sqrt{\frac{vt}{2u}+w},\varepsilon(v)\sqrt{\frac{vt}{2u}} \right)\\ 
g_v:(u,v,w,t)\mapsto\left(\sqrt{1+w-\frac{ut}{v}} ,\varepsilon(u)\sqrt{\frac{ut}{2v}},\varepsilon(v)\sqrt{\frac{ut}{2v}-w} \right)\\ 
g_w:(u,v,w,t)\mapsto\left(\sqrt{\frac{u^2-v^2}{w}},\varepsilon(u)\sqrt{\frac{1}{2}\left(1+w-\frac{u^2-v^2}{w}\right)},\varepsilon(v)\sqrt{\frac{1}{2}\left(1-w-\frac{u^2-v^2}{w}\right)}\right)\\ 
g_t:(u,v,w,t)\mapsto\left( \varepsilon(v)\sqrt{\frac{2uv}{t}},\varepsilon(t)\sqrt{\frac{1}{2}\left(1+w-\frac{2uv}{t}\right)},\sqrt{\frac{1}{2}\left(1-w-\frac{2uv}{t}\right)} \right).
\end{matrix}\right.\]

\begin{remark}
Soit $(u,v,w,t)\in P$, avec $(x,y,z)\in\s{2}$ un point vérifiant l'égalité~$p(x,y,z)=(u,v,w,t)$. Alors on a :\begin{equation}\label{eq:projective-signe}
\left\{\begin{matrix}
\varepsilon(x)\varepsilon(y)=\varepsilon(u)\\
\varepsilon(x)\varepsilon(z)=\varepsilon(v)\\
\varepsilon(u)\varepsilon(v)=\varepsilon(y)\varepsilon(z)=\varepsilon(t).
\end{matrix}\right.
\end{equation}
\end{remark}

\begin{lemma}\label{lemma:projective-inverses}
Les applications définie précédemment vérifient : \[p\circ g_i=id_{P_i},\] avec $P_i:=\{(u,v,w,t),i\neq0\}$ pour $i\in\{u,v,w,t\}$, et $P_0=\{(0,0,0,0)\}$.
\end{lemma}
\begin{proof}
Soient $(u,v,w,t)\in P$ et $(x,y,z)\in\s{2}$ tels que l'on ait l'égalité~$p(x,y,z)=(u,v,w,t)$.\begin{itemize}
    \item Supposons que $(u,v,w,t)\in P_0$, c'est à dire $(u,v,w,t)=(0,0,0,0)$. On peut vérifier que la composée des deux applications donne l'identité : \[\begin{split}
        p\circ g_0(0,0,0,0)&=p(1,0,0)\\
        &=(1\cdot0,1\cdot0,0^2-0^2,0\cdot0)\\
        &=(0,0,0,0).
    \end{split}\]
    \item Supposons $(u,v,w,t)\in P_u$, de tel sorte on peut étudier $g_u$. Avec le lemme \ref{lemma:projective-square}, ainsi que la remarque, on a : \[g_u(u,v,w,t)=\big(|x|,\,\varepsilon(u)|y|,\,\varepsilon(v)|z|\big),\]et en composant avec $p$, on aboutit à : \[\begin{split}
    p\circ g_u(u,v,w,t)&=p\big(|x|,\,\varepsilon(u)|y|,\,\varepsilon(v)|z|\big)\\
    &=\big(\varepsilon(u)|xy|,\,\varepsilon(v)|xz|,\, y^2-z^2,\,\varepsilon(t)2|yz|\big)\\
    &=(u,v,w,t).
    \end{split}\]
    \item Supposons $(u,v,w,t)\in P_v$, de telle sorte que l'on étudie $p\circ g_v$. De même qu'auparavant, on retrouve alors, en faisant la composée : \[p\circ g_v(u,v,w,t)=p\big(|x|,\,\varepsilon(u)|y|,\,\varepsilon(v)|z|\big)=(u,v,w,t).\]
    \item Supposons $(u,v,w,t)\in P_w$. De manière analogue que les cas précédents, nous aboutissons au résultat suivant : 
    \[p\circ g_w(u,v,w,t)=p\big(|x|,\,\varepsilon(u)|y|,\,\varepsilon(v)|z|\big)=(u,v,w,t).\]
    \item Supposons enfin $(u,v,w,t)\in P_t$. Ainsi, tout comme les cas précédents, on obtient le résultat suivant : \[\begin{split}
    p\circ g_t(u,v,w,t)&=p\big(\varepsilon(v)|x|,\,\varepsilon(t)|y|,\,|z|\big)\\
    &=\big(\varepsilon(u)|xy|,\,\varepsilon(v)|xz|,\, y^2-z^2,\,\varepsilon(t)2|yz|\big)\\
    &=(u,v,w,t).
    \end{split}\]
\end{itemize}
Nous avons ainsi montré que toutes les applications sont des inverses à droites de $p$.
\end{proof}

\begin{proposition}
Avec $(u,v,w,t)\in P$, nous avons les deux équations suivantes : \begin{equation}\label{eq:projective}
\left\{\begin{matrix}
2uvw=(u^2-v^2)t\\ 
(vt+2uw)(1-w)=u(t^2+2u^2).
\end{matrix}\right.
\end{equation}
\end{proposition}
Nous conjecturons le fait que $P$ est l'ensemble des solutions de ces équations.
\begin{proof}
Soit $(u,v,w,t)$ un élément de $P$, tel qu'il soit image par $p$ du point $(x,y,z)\in\s{2}$ : c'est à dire qu'il vérifie $p(x,y,z)=(u,v,w,t)$.

Pour la première équation, nous avons : \[\begin{split}
    2uvw-(u^2-v^2)t&=2(xy)(xz)(y^2-z^2)-\big((xy)^2-(xz)^2\big)(2yz)\\
    &=2x^2yz(y^2-z^2)-x^2(y^2-z^2)2yz\\
    &=0.
\end{split}\]Ensuite, pour la seconde, nous avons : \[\begin{split}
    (vt+2uw)(1-w)-u(t^2+2u^2)&=\big(2xyz^2+2xy(y^2-z^2)\big)(1-y^2+z^2)-xy\big(4y^2z^2+2x^2y^2\big)\\
    &=2xy\big(y^2(1-y^2+z^2)-y^2(2z^2+x^2)\big)\\
    &=2xy^3(1-y^2-z^2-x^2)\\
    &=0.
\end{split}\]Ce qui nous permet de prouver que les équations \eqref{eq:projective} sont bien vérifiées ici.
\end{proof}

On définit désormais l'application $g_p$, tel que sa restriction de chacun des domaines de définition donne les autres : \[\begin{split}
    g_p:P&\longrightarrow S\\
(u,v,w,t)&\mapsto\left\{\begin{matrix}
\pi_p\circ g_0(u,v,w,t) &\text{si }(u,v,w,t)=(0,0,0,0) \\ 
\pi_p\circ g_u(u,v,w,t) &\text{si }u\neq 0 \\ 
\pi_p\circ g_v(u,v,w,t) &\text{si }v\neq 0 \\ 
\pi_p\circ g_w(u,v,w,t) &\text{si }w\neq 0 \\ 
\pi_p\circ g_t(u,v,w,t) &\text{si }t\neq 0.
\end{matrix}\right.
\end{split}\]

\begin{lemma}\label{lemma:projective-inv-define}
L'application $g_p$ est bien définie. De plus, elle est une inverse à droite de $\overline{p}$.
\end{lemma}
\begin{proof}
Soit $i,j\in\{u,v,w,t\}$ tels que $i\neq j$. Supposons $ij\neq 0$. Nous avons montré dans le lemme \ref{lemma:projective-inverses} que~$p\circ g_i=p\circ g_j=id_{P_i\cap P_j}$. Or $p=\overline{p}\circ\pi_p$, ce qui nous donne : \begin{equation}\label{eq:projective-inverse}
\overline{p}\circ\pi_p\circ g_i=\overline{p}\circ\pi_p\circ g_j.
\end{equation}Par injectivité de $\overline{p}$, nous pouvons en déduire que $\pi_p\circ g_i=\pi_p\circ g_j$. De la même manière, nous pouvons montrer que $\pi_p\circ g_0=\pi_p\circ g_i$. Les intersections $P_i\cap P_j$, avec $i,j\in\{0,u,v,w,t\}$ et $i\neq j$, recouvrent l'espace $P$ en entier, donc nous avons traité tout les cas.

\bigskip De plus, du fait que les $g_i$ soient des inverses à droite de $p$, l'équation \eqref{eq:projective-inverse} vaut l'identité sur le domaine de définition de $g_i$.

Cela nous permet de dire que $g_p$ est bien définie sur $P$, comme étant inverse à droite de $\overline{p}$.
\end{proof}

Il reste à démontrer la continuité de $g_p$, que nous admettons.

\begin{theorem}
L'espace $S$ muni de la topologie quotient est homéomorphe au sous-espace $P$ de $\bb{R}^4$.
\end{theorem}
\begin{proof}
Nous avons montré dans la proposition \ref{prop:projective-relation} que l'application $\overline{p}$ est injective et continue. Dans le lemme \ref{lemma:projective-inv-define}, nous avons montré que $g_p$ était l'inverse à droite de $\overline{p}$. Finalement, avec la proposition \ref{prop:inj+inv=bij}, nous pouvons en déduire que $\overline{p}$ est une application bijective et continue. Il reste à montrer que l'inverse $g_p$ est bien continue. Par manque de temps, nous avons pas pu nous en occuper, et nous l'admettons. Voici les différents cas que l'on doit traiter pour répondre à ce problème :\begin{itemize}
    \item Nous devons montrer dans un premier temps que l'on a : \[\left \| g_p(h_u,h_v,h_w,h_t) -[1,0,0]\right \|\underset{\left \| (h_u,h_v,h_w,h_t) \right \|\rightarrow\, 0}{\longrightarrow}0,\]
    avec un cas par coordonnée.
    \item Par la suite, nous devons supposer une composante nulle, disons $t$ par exemple, et regarder $g(u_0,v_0,w_0,0)$. Nous devons traiter chacun des cas où $u_0\neq0$, puis le cas $v_0\neq0$, et $w_0\neq0$.
    \item On réitère le procédé avec $w=0$, puis $v=0$, et enfin $u=0$.
\end{itemize} Cela fait un total de 16 cas en tout à traiter.

\bigskip En admettant que l'application $\pi\circ g_p$ soit bien continue, nous pouvons alors en déduire que $\overline{p}$ réalise un homéomorphisme entre la sphère quotientée $S$ et la surface paramétrée $P$.
\end{proof}

Par transitivité de l'homéomorphisme, on peut également en conclure que l'espace quotient $C_P$ est homéomorphe au sous espace $P$ de $\bb{R}^4$.

\newpage
\appendix
\section{Plan projectif réel : homéomorphisme avec les droites de l'espace}

Dans cette partie annexe, nous allons montrer que l'espace quotient $C_P$, introduit dans la partie du plan projectif réel, est homéomorphe à l'ensemble des droites vectorielle de $\bb{R}^3$.

Dans la suite, sauf mention contraire, nous noterons $\{0\}$ le singleton contenant l'élément $(0,0,0)\in\mathbb{R}^3$.

\begin{definition}
Dans un ensemble conique $K$, on appelle \emph{relation de colinéarité} la relation d'équivalence définie par : \[\forall x,x'\in K,\quad x\eqCol x'\Leftrightarrow\exists\lambda\neq0,\ x'=\lambda x.\]

La relation de symétrie est la restriction de la relation de colinéarité à la sphère unité, car deux points sont symétriques lorsque $\lambda=\pm1$. Formalement c'est ce que nous allons montré dans la proposition \ref{homeo-RP-S}. Autrement dit, la relation d'équivalence est définie par : \[\forall x,x'\in K,\quad x\eqSym x'\Leftrightarrow x'=\pm x.\]
\end{definition}

Il en vient ainsi la définition de l'espace suivant :

\begin{definition}
On note $\mathbb{R}\mathrm{P}^2$ l'ensemble des droites vectorielles de $\mathbb{R}^3$ passant par l'origine. Cet espace est obtenu par le quotient de l'espace $\mathbb{R}^3\setminus\{0\}$ par la relation de colinéarité. Autrement dit, $\mathbb{R}\mathrm{P}^2=(\mathbb{R}^3\!\setminus\!\{0\})/\eqCol$, et est muni de la topologie quotient. Nous noterons $(x:y:z)$ les éléments de l'espace $\bb{R}\mathrm{P}^2$, avec~$(x,y,z)\in\mathbb{R}^3\setminus\{0\}$ un représentant de cette classe.
\end{definition}

Nous allons montrer que l'espace quotient $\bb{R}\mathrm{P}^2$ est homéomorphe à la demi-sphère quotienté par la relation de colinéarité $S$.

\subsection{Homéomorphismes}

\begin{proposition}\label{homeo-RP-S}
Les espaces quotients $\mathbb{R}\mathrm{P}^2$ et $S$ sont homéomorphes.
\end{proposition}
\begin{proof}
Montrons qu'il existe deux applications, réciproques l'une de l'autre, entre les deux espaces.

On considère dans un premier temps l'application définie par : \[\begin{split}
    f:\mathbb{R}^3\!\setminus\!\{0\}&\longrightarrow S\\
    (x,y,z)&\mapsto\left[\frac{1}{\sqrt{x^2+y^2+z^2}}(x,y,z)\right].
\end{split}\]

Cette application est continue. Montrons que cette application est compatible pour la relation de colinéarité.

Soit $(x,y,z)\in\mathbb{R}^3\!\setminus\!\{0\}$ et soit $\lambda\in\mathbb{R}$. On~a :\[\begin{split}
f(\lambda x,\lambda y,\lambda z)&=\left[\frac{1}{\sqrt{(\lambda x)^2+(\lambda y)^2+(\lambda z)^2}}(\lambda x,\lambda y,\lambda z)\right]\\
&=\left[\frac{\lambda}{|\lambda|\sqrt{x^2+y^2+z^2}}(x,y, z)\right]\\
&=\left[\pm\frac{1}{\sqrt{x^2+y^2+z^2}}(x,y,z)\right]\\
&=\left[\frac{1}{\sqrt{x^2+y^2+z^2}}(x,y,z)\right]\\
&=f(x,y,z).
\end{split}\]On en déduit alors que $f$ est compatible avec la relation de colinéarité, ce qui nous permet de passer au quotient grâce au théorème \ref{th:quotient}. Il existe alors une unique application continue définie par : \[\begin{split}
\overline{f}:\mathbb{R}\mathrm{P}^2&\longrightarrow S\\
(x:y:z)&\mapsto\left[\frac{1}{{\sqrt{x^2+y^2+z^2}}}\left(x,y,z\right)\right].
\end{split}\]

Cherchons désormais à construire une application réciproque. Nous considérons l'application définie par : \[\begin{split}
g:\s{2}&\longrightarrow\mathbb{R}\mathrm{P}^2\\
(x,y,z)&\mapsto(x:y:z).
\end{split}\]

Il n'y a pas de problème concernant la relation de symétrie, étant donné que c'est un cas particulier de la colinéarité. On en déduit alors directement que $g(x,y,z)=g(-x,-y,-z)$, ce qui nous permet de dire que l'application est compatible avec la relation de symétrie. De plus, cette application est continue. Ainsi, d'après le même théorème \ref{th:quotient}, il existe une unique application continue~$\overline{g}$ définie par : \[\begin{split}
\overline{g}:S&\longrightarrow\mathbb{R}\mathrm{P}^2\\
[x,y,z]&\mapsto (x:y:z).
\end{split}\]

Montrons alors que les applications $\overline{f}$ et $\overline{g}$ sont réciproques l'une de l'autre. Pour $\overline{f}\circ \overline{g}=id_{S}$, c'est évident. Soit $(x:y:z)\in\mathbb{R}\mathrm{P}^2$. On a \[\begin{split}
\overline{g}\circ\overline{f}\big((x:y:z)\big)&=g\left(\left[\frac{1}{\sqrt{x^2+y^2+z^2}}(x,y,z)\right]\right)\\
&=\left(\frac{1}{\sqrt{x^2+y^2+z^2}}(x:y:z)\right)\\
&=(x:y:z).
\end{split}\]On en conclut alors que les deux applications $\overline{f}$ et $\overline{g}$ sont réciproques l'une de l'autre et continues, ce qui fait de $\mathbb{R}\mathrm{P}^2$ et $S$ deux espaces homéomorphes.
\end{proof}

\begin{theorem}\label{th:RP2-homeo}
Les espaces quotients $\mathbb{R}\mathrm{P}^2$ et $C_P$ sont homéomorphes.
\end{theorem}
\begin{proof}
Avec le lemme \ref{homeo-RP-S}, nous avons montré que $\bb{R}\mathrm{P}^2$ était homéomorphe à la sphère quotientée $S$. Avec le lemme \ref{homeo-S-CP}, nous avons montré que cette même sphère était homéomorphe à l'espace $C_P$. Par la transitivité de l'homéomorphisme, nous pouvons en conclure que les quotients topologiques $\mathbb{R}\mathrm{P}^2$ et $C_P$ sont homéomorphes.
\end{proof}

\end{document}